\documentclass[10pt]{amsart}

\usepackage{amsmath}
\usepackage{amssymb}
\usepackage{amsfonts}%
\usepackage{graphicx}

\newtheorem{theorem}{Theorem}[section]
\newtheorem{lemma}{Lemma}[section]
\theoremstyle{definition}

\theoremstyle{remark}
\newtheorem{remark}{Remark}[section]
\numberwithin{equation}{section}
\theoremstyle{plain}

\newtheorem{proposition}{Proposition}[section]

\begin{document}
\title[Functional Stable Limit Theorem for Gibbs-Markov maps]{\textbf{A Functional Stable Limit Theorem for Gibbs-Markov maps}\\
\vspace{0.2cm} }
\author{David Kocheim}
\address{Erste Group Bank AG, Am Belvedere 1, 1100 Wien,\\
Austria }
\email{david.kocheim@erstegroup.com}
\author{Fabian P\"{u}hringer}
\address{Fakult\"{a}t f\"{u}r Mathematik, Universit\"{a}t Wien, Oskar Morgenstern Platz
1, 1090 Wien, Austria }
\email{Fabian.Puehringer@gmx.at}
\author{Roland Zweim\"{u}ller}
\address{Fakult\"{a}t f\"{u}r Mathematik, Universit\"{a}t Wien, Oskar Morgenstern Platz
1, 1090 Wien, Austria}
\email{roland.zweimueller@univie.ac.at}
\thanks{\textit{Acknowledgement.} F.P. gratefully acknowledges support through
FWF-grant Y00782. R.Z. thanks Jon Aaronson, Alexey Korepanov and Ian Melbourne
for inspiring discussions related to this topic. We are also grateful to the
referee whose comments helped to improve the presentation.}
\subjclass[2010]{ Primary 28D05, 37A25, 37C30, 11K50.}
\keywords{Weak invariance principle, stable laws, stable L\'{e}vy processes, weakly
dependent processes, stationary sequences.}

\begin{abstract}
For a class of locally (but not necessarily uniformly) Lipschitz continuous
$d$-dimensional observables over a Gibbs-Markov system, we show that
convergence of (suitably normalized and centered) ergodic sums to a
non-Gaussian stable vector is equivalent to the distribution belonging to the
classical domain of attraction, and that it implies a weak invariance
principle in the (strong) Skorohod $\mathcal{J}_{1}$-topology on
$\mathcal{D}([0,\infty),\mathbb{R}^{d})$. The argument uses the classical
approach via finite-dimensional marginals and $\mathcal{J}_{1}$-tightness. As
applications, we record a Spitzer-type arcsine law for certain $\mathbb{Z}%
$-extensions of Gibbs-Markov systems, and prove an asymptotic independence
property of excursion processes of intermittent interval maps.

\end{abstract}
\maketitle

\section{Introduction}

The study of ergodic dynamical systems naturally leads to many questions about
stationary sequences with nontrivial dependence structure. If $T$ is a map
preserving a probability measure $\mu$ on $(X,\mathcal{A})$, and
$f:X\rightarrow\mathbb{R}^{d}$ is some \emph{observable} (that is, a
measurable function), the stationary sequence $(f\circ T^{k})_{k\geq0}$ on
$(X,\mathcal{A},\mu)$ is expected to exhibit properties similar to those of
iid sequences as soon as $T$ possesses sufficiently strong mixing properties,
and $f$ is regular enough. In particular, the sequence of \emph{ergodic sums}
$\mathbf{S}_{n}(f):=\sum_{k=0}^{n-1}f\circ T^{k}$ should behave like a
classical partial sum process, for which, among a multitude of other results,
functional limit theorems are available. Indeed, there is a well developed
theory clarifying the asymptotic behavior of such sequences in situations with
a Gaussian limit, see for example \cite{MPU}.\newline

In the present article we are interested in situations where $f$ has a heavy
tail, and its distribution $\mu\circ f^{-1}$ is in the \emph{domain of
attraction} \emph{of some non-Gaussian }$\alpha$\emph{-stable random vector}
$S$. (We refer to \cite{MS} and \cite{ST} for background information on stable
laws and processes, see also Section 2 below.) Then there exist two sequences
of constants $A_{n}\in\mathbb{R}^{d}$ and $B_{n}>0$, $n\geq1$, such that
distributional convergence%
\begin{equation}
\frac{1}{B_{n}}\,(\mathbf{S}_{n}(f)-A_{n})\Longrightarrow S\text{ \quad as
}n\rightarrow\infty\label{Eq_ADADAD}%
\end{equation}
takes place provided that $(f\circ T^{k})_{k\geq0}$ is an iid sequence. In the
classical iid setup, the \emph{Stable Limit Theorem (SLT)} (\ref{Eq_ADADAD})
automatically entails a \emph{Functional Stable Limit Theorem (FSLT)} (or
\emph{weak invariance principle}) in $(\mathcal{D}([0,\infty),\mathbb{R}%
^{d}),\mathcal{J}_{1})$, which asserts distributional convergence of the
\emph{partial sum processes} $\mathsf{S}^{[n]}$ given by
\begin{equation}
\mathsf{S}^{[n]}:X\rightarrow\mathcal{D}([0,\infty),\mathbb{R}^{d})\text{,
}n\geq1\text{, \quad}\mathsf{S}_{t}^{[n]}:=\frac{1}{B_{n}}\left(
\mathbf{S}_{\left\lfloor tn\right\rfloor }(f)-\frac{\left\lfloor
tn\right\rfloor }{n}A_{n}\right)  \text{,}\label{Eq_DefPartialSumProc}%
\end{equation}
to the $\alpha$-stable L\'{e}vy process $\mathsf{S}=(\mathsf{S}_{t})_{t\geq0}$
with $\mathsf{S}_{1}$ distributed like $S$ (see \cite{Sk}, \cite{Sk2}),
\begin{equation}
\mathsf{S}^{[n]}\Longrightarrow\mathsf{S}\text{ \quad in }(\mathcal{D}%
([0,\infty),\mathbb{R}^{d}),\mathcal{J}_{1})\text{.}\label{Eq_FSLT}%
\end{equation}
Here, $\mathcal{D}([0,\infty),\mathbb{R}^{d})$ is the \emph{Skorohod}
\emph{space} of right-continuous functions $\mathsf{x}:[0,\infty
)\rightarrow\mathbb{R}^{d}$ possessing left limits everywhere, and we always
use the \emph{Skorohod }$\mathcal{J}_{1}$\emph{-topology} (details
below).\newline

It is known that for dependent stationary sequences the $\mathcal{J}_{1}%
$\emph{-}FSLT does not, in general, follow from the SLT, not even if $d=1$
(see Examples 1.1 and 2.1 in \cite{T2}).

The main result of the present article, Theorem \ref{T_SLTgivesFSLT} below,
shows that a SLT (\ref{Eq_ADADAD}) for sufficiently regular vector-valued
observables over a \emph{Gibbs-Markov system} implies the corresponding FSLT
(\ref{Eq_FSLT}) in the $\mathcal{J}_{1}$-topology. The regularity condition we
use is that the tail of the cylinderwise Lipschitz constant should not be
heavier than the tail of $f$ itself.

Gibbs-Markov maps form an important basic class of systems. In this context,
SLTs have for example been established in \cite{AD}, \cite{Goue}, \cite{Goue2}
for certain real-valued observables $f$ which are Lipschitz\ on cylinders.
Based on work in \cite{T1}, the article \cite{T2} proves, for $d=1$, a $\mathcal{J}_{1}%
$\emph{-}FSLT (\ref{Eq_FSLT}) in certain dynamical situations, including that
of piecewise constant observables on a Gibbs-Markov system. 
In the case of stable laws of index $\alpha \in (1,2)$, a vector-valued FSLT
similar to our Theorem \ref{T_SLTgivesFSLT} is given in Section 4 of \cite{CFKM}.
This, too, is based on \cite{T1}. Our present result covers all $\alpha \in (0,2)$,
and the proof is independent of \cite{T1}.

Our Theorem \ref{T_FSL_GM} (which rephrases the $d=1$ case of Theorem \ref{T_SLTgivesFSLT}
in easily applicable form) gives a $\mathcal{J}_{1}$\emph{-}FSLT for a 
large class of real observables. (More on the relation to the FSLT\ of \cite{T2} in
Remark \ref{Rem_Marta} below.) The second special case of Theorem \ref{T_SLTgivesFSLT} we make explicit,
Theorem \ref{T_FSL_GMddim}, asserts that, for a vector-valued observable
$f=(f^{(1)},\ldots,f^{(d)})$ with individual components $f^{(j)}$ of said
regularity and with asymptotically proportional tails in the domains of
attraction of $\alpha$-stable laws, the partial sum processes converge, under
the $\mathcal{J}_{1}$-topology, to an independent tuple of scalar stable
L\'{e}vy motions, provided that the tails of the $f^{(j)}$ are determined on
non-overlapping sets. 

In a final section, we illustrate the use of these results. First, we apply
Theorem \ref{T_FSL_GM} in the setup of certain infinite measure preserving
skew products with Gibbs-Markov base to obtain a Spitzer-type arcsine law
(Theorem \ref{T_SpitzerTypeAsinus}). Second, for prototypical interval maps on
$[0,1]$ with two indifferent fixed points of the same order, at $x=0$ and
$x=1$, our Theorem \ref{T_FSL_GMddim} entails asymptotic independence between
the two processes of excursions to small neighbourhoods of $x=0$ and $x=1$,
respectively. Here, the maps may have finite or infinite invariant measure. We
require the fixed points to be so strong that those excursion processes have
non-Gaussian limits, see Theorem \ref{T_IntermittExcursionProc}.\newline

The approach used in \cite{T1}, \cite{T2} was to study the asymptotic
behaviour of the point processes which capture the occurrences of large
individual observations, thus proving convergence to the limit process via
L\'{e}vy-It\^{o}-type representations. In contrast, the present article
follows the classical \textquotedblleft convergence of marginals plus
tightness\textquotedblright\ approach.

\section{Background and Main results}

We begin by fixing notations and collecting the required background material.
For functions $\tau_{1},\tau_{2}$ and $c$ a constant, we write $\tau
_{1}(t)\sim c\tau_{2}(t)$ as $t\rightarrow\infty$ to indicate, as in
\cite{BGT}, that $\tau_{2}(t)>0$ for large $t$ and $\tau_{1}(t)/\tau
_{2}(t)\rightarrow c$, even in case $c=0$.\ \newline\newline%
\textbf{Distributional convergence.} If $R_{l}$, $l\geq1$, are Borel
measurable maps of $(X,\mathcal{A})$ into some metric space $(\mathfrak{E}%
,d_{\mathfrak{E}})$, while $\nu_{l}$, $l\geq1$, are probability measures on
$(X,\mathcal{A})$, and $R$ is another \emph{random element} of $\mathfrak{E}$
(defined on some $(\Omega,\mathcal{F},\Pr)$), then we write
\begin{equation}
R_{l}\overset{\nu_{l}}{\Longrightarrow}R\text{ \quad as }l\rightarrow\infty
\end{equation}
to indicate that $\nu_{l}\circ R_{l}^{-1}\Longrightarrow\Pr\circ R^{-1}$ (the
usual weak convergence of probability measures, \cite{Bi}). This is
\emph{distributional convergence} to $R$ of the $R_{l}$ when the latter
functions are regarded as random variables on the probability spaces
$(X,\mathcal{A},\nu_{l})$, respectively. It includes the case of a single
measure $\nu$, where $R_{l}\overset{\nu}{\Longrightarrow}R$ means that the
distributions $\nu\circ R_{l}^{-1}$ of the $R_{l}$ under $\nu$ converge weakly
to the law of $R$. If $\nu$ is understood, we may simply write $R_{l}%
\Longrightarrow R$.

The limit theorems we are to discuss are in fact instances of \emph{strong
distributional convergence} with respect to the invariant measure $\mu$
(terminology taken from \cite{A}) or \emph{mixing limit theorems}, meaning
that convergence $\overset{\nu}{\Longrightarrow}$ with respect to one
probability measure $\nu\ll\mu$ implies convergence w.r.t. all such measures,
see Theorems \ref{T_SLTgivesFSLT}, \ref{T_FSL_GM}, \ref{T_FSL_GMddim} and
Lemma \ref{L_StrongDistCge} below.\newline\newline\textbf{Stable random
vectors.} A random vector $S=(S^{(1)},\ldots,S^{(d)})$ (or its law) is
\emph{stable} if for any positive numbers $a,b$ there are constants $c>0$ and
$D\in\mathbb{R}^{d}$ such that $aS_{1}+bS_{2}\overset{d}{=}cS+D$, where
$S_{1},S_{2}$ are independent copies of $S$ and $\overset{d}{=}$ indicates
equality of distributions. This is equivalent to the assertion that there is
some $\alpha\in(0,2]$ such that for any $n\geq2$ there is some constant
$D_{n}\in\mathbb{R}^{d}$ for which $S_{1}+\ldots+S_{n}\overset{d}%
{=}n^{1/\alpha}S+D_{n}$, whenever $S_{1},S_{2},\ldots$ are independent copies
of $S$. In this case $S$ (or its law) is said to be $\alpha$\emph{-stable},
and $\alpha$ is the \emph{index} of $S$ (see \S 2.1 of \cite{ST}). $S$ is a
Gaussian iff $\alpha=2$. Stable vectors are exactly those random elements $S$
of $\mathbb{R}^{d}$ which occur as limits%
\begin{equation}
\frac{1}{B_{n}}\,\left(  \sum_{k=0}^{n-1}Z_{k}-A_{n}\right)  \Longrightarrow
S\text{ \quad as }n\rightarrow\infty\text{,}\label{Eq_ClassicalSLTNew}%
\end{equation}
for iid sequences $(Z_{k})_{k\geq0}$ of random vectors and constants $A_{n}%
\in\mathbb{R}^{d}$ and $B_{n}>0$. In this case, $Z_{0}$ (or its law $Q$) is
said to belong to the \emph{domain of attraction of }$S$, which we will
indicate by writing $Z_{0}$ (or $Q$) $\in\mathrm{DOA}(S)$ (see \S 7.3 in
\cite{MS}).

We call a Borel measure on $\mathbb{R}^{d}$ \emph{full} if it is not supported
on any proper affine subspace of $\mathbb{R}^{d}$. A $d$-dimensional random
vector is \emph{full} if its law is. The collection of all full stable random
vectors $S$ with index $\alpha\in(0,2)$ can be represented as a parametrized
family $\{S_{\alpha}(\Lambda,c)\}_{c\in\mathbb{R}^{d},\Lambda\in\Sigma}$ where
$c$ is the \emph{center} of $S_{\alpha}(\Lambda,c)$ and $\Lambda$ is its
\emph{spectral measure}. Here $\Sigma$ is the family of all finite full Borel
measures $\Lambda$ on $\mathbb{R}^{d}$ which are concentrated on the unit
sphere $\mathbb{S}^{d-1}$ (see Theorem 7.3.16 of \cite{MS}). The constant
vector $c$ is a simple location parameter such that $S_{\alpha}(\Lambda
,c)\overset{d}{=}S_{\alpha}(\Lambda,0)+c$, while $\Lambda$ encodes the tail
behaviour of $S=S_{\alpha}(\Lambda,c)$ in that
\begin{gather}
\Pr\left[  \left.  \frac{S}{\left\Vert S\right\Vert }\in D\,\right\vert
\,\left\Vert S\right\Vert >t\right]  \longrightarrow\frac{\Lambda(D)}%
{\Lambda(\mathbb{S}^{d-1})}\text{ \quad as }t\rightarrow\infty\text{,}%
\label{Eq_RegVarOfStableTails}\\
\text{for Borel }D\subseteq\mathbb{S}^{d-1}\text{ whose boundary \emph{in}
}\mathbb{S}^{d-1}\text{ is a null set for }\Lambda\text{.}\nonumber
\end{gather}
Moreover, as shown in \cite{R} (see also Theorem 8.2.18 of \cite{MS}), the
(law of a) random vector $Z$ belongs to $\mathrm{DOA}(S_{\alpha}(\Lambda,c))$
iff $V(t):=\Pr[\left\Vert Z\right\Vert >t]$, $t\geq0$, is regularly varying of
index $-\alpha$, and (\ref{Eq_RegVarOfStableTails}) holds with $S $ replaced
by $Z$.

In the special case of a scalar function $f:X\rightarrow\mathbb{R}$,
measurable on a probability space $(X,\mathcal{A},\mu)$, we thus see that its
distribution $\mu\circ f^{-1}$ is in the domain of attraction of some
non-Gaussian $\alpha$-stable random variable iff there are $\alpha\in(0,2) $,
a \emph{slowly varying} \emph{function} $\ell$ (that is, positive and
measurable with $\ell(\rho s)/\ell(s)\rightarrow1$ as $s\rightarrow\infty$ for
every $\rho>0$), and constants $c_{+},c_{-}\geq0$ with $c_{+}+c_{-}>0$ such
that, as $t\rightarrow\infty$,
\begin{equation}
\mu(f>t)=(c_{+}+o(1))t^{-\alpha}\ell(t)\text{\quad and}\quad\mu(f\leq
-t)=(c_{-}+o(1))t^{-\alpha}\ell(t)\text{.}\label{Eq_DomainAttr}%
\end{equation}
The law of a specific variable $S$ which partial sums are then attracted to is
given by the Fourier transform
\begin{equation}
\mathbb{E}[e^{itS}]=e^{-c_{\alpha}\underline{\beta}\,\left\vert t\right\vert
^{\alpha}(1-i\beta\,\mathrm{sgn}(t)\,\omega(\alpha,t))}\text{, \quad}%
t\in\mathbb{R}\text{, }\label{Eq_TheFourierTransformOfG}%
\end{equation}
where $c_{\alpha}:=\Gamma(1-\alpha)\cos(\alpha\pi/2)$ if $\alpha\neq1 $ and
$c_{\alpha}:=\pi/2$ if $\alpha=1$, while $\underline{\beta}:=c_{+}+c_{-}$,
$\overline{\beta}:=c_{+}-c_{-}$ and $\beta:=\overline{\beta}/\underline{\beta
}$, while $\omega_{\alpha}(t):=\tan(\alpha\pi/2)$ if $\alpha\neq1$ and
$\omega_{\alpha}(t):=-(2/\pi)\log\left\vert t\right\vert $ if $\alpha=1$.
While the limit laws of partial sums are only unique up to type, we will take
the limit to be this particular variable $S$. (This is the convention used in
\cite{AD}, \cite{Goue}, \cite{Goue2}.)

In this situation, partial sums of iid sequences $(Z_{k})$ with the same
distribution as $f$ satisfy an SLT (\ref{Eq_ClassicalSLTNew}) with sequences
$(A_{n})$ and $(B_{n})$ defined in terms of this distribution. Specifically,
one can take $(B_{n})$ so that
\begin{equation}
n\ell(B_{n})=B_{n}^{\alpha}\text{ \quad for }n\geq1\text{,}\label{Eq_DefBn}%
\end{equation}
and $A_{n}=0$ if $\alpha<1$, $A_{n}=n\int f\,d\mu$ if $\alpha>1$, while the
definition of $A_{n}$ is more complicated in case $\alpha=1$, see Section 6 of
\cite{AD}. In either case,
\begin{equation}
A_{n}=o(nB_{n})\text{ \quad as }n\rightarrow\infty\text{.}\label{Eq_Dings}%
\end{equation}
We shall refer to this choice of $(A_{n},B_{n})_{n\geq1}$ in $\mathbb{R}%
\times(0,\infty)$ as the \emph{canonical normalizing sequence for}
$[\alpha,c_{+},c_{-}]$, and to $S$ as the \emph{canonical limit (law) for
}$[\alpha,c_{+},c_{-}]$. (In the notation of \S 1.1 in \cite{ST}, we have
$S\overset{d}{=}S([c_{\alpha}(c_{+}-c_{-})]^{1/\alpha},(c_{+}-c_{-}%
)/(c_{+}+c_{-}),0)$.) \newline\newline\textbf{Gibbs-Markov systems.} A
\emph{piecewise invertible probability preserving system} is a tuple
$(X,\mathcal{A},\mu,T,\xi)$ where $T:X\rightarrow X$ \ is a measure preserving
map on the probability space $(X,\mathcal{A},\mu)$, and $\xi\subseteq
\mathcal{A}$ is a countable partition (mod $\mu$) of $X$ with $\mathcal{A}%
=\sigma(T^{-k}\xi:k\geq0)$ (mod $\mu$) and such that the restriction of $T$ to
any \emph{cylinder} $Z\in\xi$ is a measurably invertible map $T\mid
_{Z}:Z\rightarrow TZ$ with inverse $v_{Z}:TZ\rightarrow Z$. The partition and
the system are said to be \emph{Markov} if for each $Z\in\xi$ the image $TZ$
is measurable $\xi$ (mod $\mu$). It has the \emph{big image property} if
$\inf_{Z\in\xi}\mu(TZ)>0$. Write $\xi_{n}:=%
{\textstyle\bigvee\nolimits_{k=0}^{n-1}}
T^{-k}\xi$ for the family of \emph{cylinders of rank} $n\geq1$. We denote the
element of $\xi_{n}$ containing $x$ by $\xi_{n}(x):=\bigcap_{k=0}^{n-1}%
\xi(T^{k}x)$ (well defined for a.e. $x$).

The \emph{separation time} of two points $x,y\in X$ is $s(x,y):=\inf
\{n\geq1:\xi_{n}(x)\neq\xi_{n}(y)\}$. For a parameter $\theta\in(0,1)$ define
a \emph{dynamical metric} by letting $d_{\theta}(x,y):=\theta^{s(x,y)}$.
Evidently, $T$ is \emph{uniformly expanding} w.r.t. $d_{\theta}\ $in that
$d_{\theta}(x,y)=\theta$ $d_{\theta}(Tx,Ty)$ a.e. For $f:X\rightarrow
\mathfrak{F}$ (with $(\mathfrak{F},d_{\mathfrak{F}})$ some metric space) and
$W\subseteq X$ set $D_{W}(f):=\inf\{L>0:d_{\mathfrak{F}}(f(x),f(y))\leq
L\,d_{\theta}(x,y)$ for $x,y\in W\}$, the least Lipschitz constant of $f$ on
the set $W$, and $D_{\mathcal{W}}(f):=\sup_{W\in\mathcal{W}}D_{W}(f)$ if
$\mathcal{W}$ is a collection of sets. Call $f$ \emph{uniformly piecewise
Lipschitz} in case $D_{\xi}(f)<\infty$.

Consider the Radon-Nikodym derivatives $v_{Z}^{\prime}:TZ\rightarrow
\lbrack0,\infty)$, $Z\in\xi$, with $v_{Z}^{\prime}:=d(\mu\circ v_{Z})/d\mu$. A
Markov system $(X,\mathcal{A},\mu,T,\xi)$ with the big image property is said
to be \emph{Gibbs-Markov} if, in addition, $\sup_{Z\in\xi}D_{TZ}(\log\circ
v_{Z}^{\prime})<\infty$ (for suitable versions of the a.e. defined functions
$v_{Z}^{\prime}$). In this case a routine argument shows that the system has
\emph{bounded distortion} in that there is some $R\in\lbrack0,\infty)$ such
that
\begin{equation}
e^{-R}\,\frac{\mu(T^{n}Z\cap E)}{\mu(T^{n}Z)}\leq\frac{\mu(Z\cap T^{-n}E)}%
{\mu(Z)}\leq e^{R}\,\frac{\mu(T^{n}Z\cap E)}{\mu(T^{n}Z)}\text{ \quad}%
\begin{array}
[c]{l}%
\text{whenever }n\geq1\text{, }\\
Z\in\xi_{n}\text{, and }E\in\mathcal{A}\text{.}%
\end{array}
\label{Eq_BddDistortion}%
\end{equation}

For observables $f$ taking values in some normed space $(\mathfrak{F}%
,\left\Vert \centerdot\right\Vert )$, the \emph{ergodic sums} will be denoted
$\mathbf{S}_{n}(f):=\sum_{k=0}^{n-1}f\circ T^{k}$, $n\geq1$. If $f$ is
understood, we will simply abbreviate $S_{n}:=\mathbf{S}_{n}(f)$. To deal with
observables $f:X\rightarrow\mathfrak{F}$ which are not necessarily
\emph{uniformly} piecewise Lipschitz, we consider the associated $\xi
$-measurable function
\begin{equation}
\vartheta_{f}:X\rightarrow\lbrack0,\infty]\text{, \quad}\vartheta_{f}:=%
{\textstyle\sum\nolimits_{Z\in\xi}}
D_{Z}(f)1_{Z}%
\end{equation}
(defined almost everywhere), which collects the best Lipschitz constants on
individual rank-one cylinders. If $\vartheta_{f}$ is unbounded, the decay rate
of its tail $\mu(\vartheta_{f}>t)$ as $t\rightarrow\infty$ provides a
meaningful way of quantifying the overall regularity of the function. The
stable limit theorems for real-valued observables obtained in \cite{AD}, which
assume boundedness of $\vartheta_{f}$, have been extended significantly in
\cite{Goue2}, where it is shown that for $f:X\rightarrow\mathbb{R}$ in the
domain of attraction of a stable random variable $S$, the assumption that
$\int\vartheta_{f}^{\eta}\,d\mu<\infty$ for some $\eta\in(0,1]$ is sufficient
for the SLT (\ref{Eq_ADADAD}).\newline\newline\textbf{Link to the iid case.}
The result just quoted is based on the insight (Theorem 1.5 of \cite{Goue2})
that (excluding the square-integrable Gaussian case) for such $f$ the SLT for
the dynamical system holds iff it holds for the partial sums of an iid
sequence of random variables with the same distribution as $f$. This easily
extends from the scalar case studied in \cite{Goue2} to the situation of
$d$-dimensional random vectors.

\begin{proposition}
[\textbf{Nondegenerate distribtional limits - GM versus iid}]%
\label{Prop_GMvsIID}Let $(X,\mathcal{A},\mu,T,\xi)$ be a mixing probability
preserving Gibbs-Markov system, and $f=(f^{(1)},\ldots,f^{(d)}):X\rightarrow
\mathbb{R}^{d}$ an observable satisfying $\int\vartheta_{f}^{\eta}%
\,d\mu<\infty$ for some $\eta\in(0,1]$. Suppose that $S$ is a full
$d$-dimensional random vector, and $(A_{n},B_{n})_{n\geq1}$ a sequence in
$\mathbb{R}^{d}\times(0,\infty)$ such that $\sqrt{n}=o(B_{n})$ as
$n\rightarrow\infty$.

Let $(Z_{k})_{k\geq0}$ be an iid sequence with $Z_{0}\overset{d}{=}f$. Then
\begin{equation}
\frac{1}{B_{n}}\,(\mathbf{S}_{n}(f)-A_{n})\overset{\mu}{\Longrightarrow
}S\text{ \quad as }n\rightarrow\infty\label{Eq_dSLT_G}%
\end{equation}
iff
\begin{equation}
\frac{1}{B_{n}}\,\left(  \sum_{k=0}^{n-1}Z_{k}-A_{n}\right)  \Longrightarrow
S\text{ \quad as }n\rightarrow\infty\text{.}\label{Eq_dSLT_i}%
\end{equation}
If (\ref{Eq_dSLT_G}) and (\ref{Eq_dSLT_i}) hold, then $S$ is a stable random
vector. Specifically, if $S$ is $\alpha$-stable, $\alpha$ $\in(0,2)$, then
$(B_{n})$ is regularly varying of index $1/\alpha$, while each $\tau_{f^{(i)}%
}(s):=\mu(\left\vert f^{(i)}\right\vert >s)$ is regularly varying of index
$-\alpha$ and satisfies%
\begin{equation}
n\,\tau_{f^{(i)}}(B_{n})\rightarrow c^{(i)}\text{ \quad as }n\rightarrow
\infty\label{Eq_shvfbjsfdbvjhdbbbbbbyyyyyyyyyyy2}%
\end{equation}
for some $c^{(i)}\in(0,\infty)$. Moreover, the $A_{n}=(A_{n}^{(1)}%
,\ldots,A_{n}^{(d)})$ satisfy
\begin{equation}
A_{n}^{(i)}=o(nB_{n})\text{ \quad as }n\rightarrow\infty\text{.}%
\label{Eq_hjjhjhjhjhjhjjjjjjjjjjjjjjjjjjjjjjjjjjjjjjjjjjj2}%
\end{equation}

\end{proposition}%

\vspace{0.2cm}%

\begin{proof}
\textbf{(i)}\ Take any non-zero linear form $\psi:\mathbb{R}^{d}%
\rightarrow\mathbb{R}$. Due to $\vartheta_{\psi\circ f}\leq\left\Vert
\psi\right\Vert \vartheta_{f}$ we have $\int\vartheta_{\psi\circ f}^{\eta
}\,d\mu<\infty$. Hence the function $\psi\circ f:X\rightarrow\mathbb{R}$
belongs to the family of scalar observables studied in \cite{Goue2}.

Assume (\ref{Eq_dSLT_i}). Then by linearity and continuity of $\psi$,
\begin{equation}
\frac{1}{B_{n}}\,\left(  \sum_{k=0}^{n-1}\psi(Z_{k})-\psi(A_{n})\right)
\Longrightarrow\psi(S)\text{ \quad as }n\rightarrow\infty\text{,}%
\label{Eq_dSLT_psi_i}%
\end{equation}
and the classical CLT shows that $\int(\psi\circ f)^{2}\,d\mu=\mathbb{E}%
[\psi(Z_{k})^{2}]=\infty$ since $\sqrt{n}=o(B_{n})$ and $\psi(S)$ is
non-degenerate. On the other hand, if we assume (\ref{Eq_dSLT_G}), then the
$L_{2}(\mu)$-case of Theorem 1.5 in \cite{Goue2} shows that $\int(\psi\circ
f)^{2}\,d\mu=\infty$. Consequently, either of (\ref{Eq_dSLT_G}) and
(\ref{Eq_dSLT_i}) implies that the infinite variance case of Theorem 1.5 in
\cite{Goue2} applies to $\psi\circ f$. The latter states that
(\ref{Eq_dSLT_psi_i}) holds iff
\begin{equation}
\frac{1}{B_{n}}\,\left(  \mathbf{S}_{n}(\psi\circ f)-\psi(A_{n})\right)
\Longrightarrow\psi(S)\text{ \quad as }n\rightarrow\infty\text{.}%
\label{Eq_dSLT_psi_G}%
\end{equation}
\textbf{(ii)} By Cram\'{e}r-Wold, (\ref{Eq_dSLT_G}) is equivalent to the
statement that (\ref{Eq_dSLT_psi_G}) holds for all $\psi$, while
(\ref{Eq_dSLT_i}) is equivalent to the assertion that (\ref{Eq_dSLT_psi_i}) is
valid for all $\psi$. Therefore equivalence of (\ref{Eq_dSLT_G}) and
(\ref{Eq_dSLT_i}) follows from step (i).\newline\newline\textbf{(iii)} As
mentioned before, (\ref{Eq_dSLT_i}) implies stability of $S $. Write
$S=(S^{(1)},\ldots,S^{(d)})$, and fix any $i\in\{1,\ldots,d\}$. As a special
case of (i) we see that $f^{(i)}$ is in the domain of attraction of $S^{(i)}$,
and in view of the explicit description (\ref{Eq_DomainAttr}) of scalar
domains of attraction above, we conclude that the tail $\tau_{f^{(i)}}$ of
$f^{(i)}$ is regularly varying of index $-\alpha$. Moreover, $(A_{n}%
^{(i)},B_{n})$ is a normalizing sequence for the partial sums of an iid
sequence of variables distributed like $f^{(i)}$. By the standard
one-dimensional convergence of types theorem (for example, Theorem II.10.2 of
\cite{GK}), $(B_{n})$ here is asymptotically proportional to the canonical
normalizing sequence in (\ref{Eq_DefBn}), and hence regularly varying of index
$1/\alpha$. Combining (\ref{Eq_DomainAttr}) and (\ref{Eq_DefBn}) we get
(\ref{Eq_shvfbjsfdbvjhdbbbbbbyyyyyyyyyyy2}). Regarding $(A_{n}^{(i)})$,
convergence of types shows that (since $B_{n}\rightarrow\infty$), statement
(\ref{Eq_hjjhjhjhjhjhjjjjjjjjjjjjjjjjjjjjjjjjjjjjjjjjjjj2}) follows from the
corresponding statement (\ref{Eq_Dings}) for the canonical normalization.
\end{proof}%

\vspace{0.2cm}%

\noindent
\textbf{Skorohod spaces.} Let $(\mathfrak{F},d_{\mathfrak{F}})$ be a metric
space. Recall that $\mathcal{D}([0,1],\mathfrak{F})$ is the space of
right-continuous real functions $\mathsf{x}:[0,1]\rightarrow\mathfrak{F}$
possessing left limits everywhere (\emph{cadlag} functions). Denote functions
to be regarded as elements of the path space $\mathcal{D}([0,1],\mathfrak{F})$
by $\mathsf{x},\mathsf{y}$ and their values at $t$ by $\mathsf{x}%
_{t},\mathsf{y}_{t}$. Equip the space $\mathcal{D}([0,1],\mathfrak{F})$ with
the standard \emph{Skorohod }$\mathcal{J}_{1}$\emph{-topology} (see \cite{Bi},
\cite{GS}, \cite{Sk}, or \cite{W}). Two functions $\mathsf{x},\mathsf{y}%
\in\mathcal{D}([0,1],\mathfrak{F})$ are close in this topology if they are
uniformly close after a small distortion of the domain. Formally, let
$\Lambda$ be the set of increasing homeomorphisms $\lambda:[0,1]\rightarrow
\lbrack0,1]$, and let $\lambda_{id}\in\Lambda$ denote the identity. Then
$d_{\mathcal{J}_{1}}(\mathsf{x},\mathsf{y})=d_{\mathcal{J}_{1},1}%
(\mathsf{x},\mathsf{y}):=\inf_{\lambda\in\Lambda}\{\sup_{[0,1]}d_{\mathfrak{F}%
}(\mathsf{x}\circ\lambda,\mathsf{y})\vee\sup_{\lbrack0,1]}\left\vert
\lambda-\lambda_{id}\right\vert \}$ defines a metric on $\mathcal{D}%
([0,1],\mathfrak{F})$ which induces the $\mathcal{J}_{1}$-topology. While its
restriction to $\mathcal{C}([0,1],\mathfrak{F})$ coincides with the uniform
topology, discontinuous functions are $\mathcal{J}_{1}$-close to each other
only if they have similar jumps at nearby positions.

For any $s>0$ the space $(\mathcal{D}([0,s],\mathfrak{F}),\mathcal{J}_{1})$ of
cadlag functions on $[0,s]$ is defined analogously, with metric
$d_{\mathcal{J}_{1},s}$. To obtain the $\mathcal{J}_{1}$-topology on
$\mathcal{D}([0,\infty),\mathfrak{F})$, the space of all cadlag functions
$\mathsf{x}:[0,\infty)\rightarrow\mathfrak{F}$, we use $d_{\mathcal{J}%
_{1},\infty}:=\int_{0}^{\infty}e^{-s}(1\wedge d_{\mathcal{J}_{1},s})\,ds$.
Then, convergence $\mathsf{x}^{[n]}\rightarrow\mathsf{x}$ as $n\rightarrow
\infty$ in $(\mathcal{D}([0,\infty),\mathfrak{F}),\mathcal{J}_{1})$ means
$\mathsf{x}^{[n]}\rightarrow\mathsf{x}$ in $(\mathcal{D}([0,s],\mathfrak{F}%
),\mathcal{J}_{1})$ for every continuity point $s>0$ of $\mathsf{x}$ (we
simply denote the restriction of $\mathsf{x}\in\mathcal{D}([0,\infty
),\mathfrak{F})$ to $[0,s]$ by $\mathsf{x}$ again).

We will study situations in which $\mathfrak{F}=\mathbb{R}$ or $\mathbb{R}%
^{d}$. In the latter case, the standard $\mathcal{J}_{1}$-topology on
$\mathcal{D}([0,\infty),\mathbb{R}^{d})$ is sometimes called the \emph{strong}
$\mathcal{J}_{1}$-topology, in order to distinguish it from the \emph{weak}
$\mathcal{J}_{1}$-topology on $\mathcal{D}([0,\infty),\mathbb{R}^{d})$ given
by componentwise convergence in $\mathcal{D}([0,\infty),\mathbb{R})$, see
Section 3.3 of \cite{W}. The random elements of $\mathcal{D}([0,1],\mathbb{R}%
^{d})$ or $\mathcal{D}([0,\infty),\mathbb{R}^{d})$ of interest in this paper
will be denoted by $\mathsf{S},\mathsf{S}^{[n]},\ldots$ with corresponding
coordinate variables $\mathsf{S}_{t},\mathsf{S}_{t}^{[n]}\ $etc.\newline%
\newline\newline\newline\textbf{SLT implies the }$\mathcal{J}_{1}%
$\textbf{-FSLT.} The core result of the present paper shows that an SLT for
$d$-dimensional observables of reasonable regularity over a Gibbs-Markov
system automatically entails a functional version in the (strong)
$\mathcal{J}_{1}$-topology of $\mathcal{D}([0,\infty),\mathbb{R}^{d})$.

\begin{theorem}
[\textbf{SLT implies }$\mathcal{J}_{1}$\textbf{-FSLT for Gibbs-Markov maps}%
]\label{T_SLTgivesFSLT}Let $(X,\mathcal{A},\mu,T,\xi)$ be a mixing probability
preserving Gibbs-Markov system, and $f:X\rightarrow\mathbb{R}^{d}$ an
observable satisfying $\mu(\vartheta_{f}>t)=O(\mu(\left\Vert f\right\Vert
>t))$. Assume that $\mu\circ f^{-1}\in\mathrm{DOA}(S)$ for some full $\alpha
$-stable random vector $S$ in $\mathbb{R}^{d}$ with $\alpha\in(0,2)$, so that
there are $(A_{n},B_{n})\in\mathbb{R}^{d}\times(0,\infty)$, $n\geq1$, such
that
\begin{equation}
\frac{1}{B_{n}}\,(\mathbf{S}_{n}(f)-A_{n})\overset{\mu}{\Longrightarrow
}S\text{ \quad as }n\rightarrow\infty\text{.}\label{Eq_BasicAssmThmImpli}%
\end{equation}
Then the partial sum processes $\mathsf{S}^{[n]}=(\mathsf{S}_{t}^{[n]}%
)_{t\geq0}$ given by
\begin{equation}
\mathsf{S}^{[n]}:X\rightarrow\mathcal{D}([0,\infty),\mathbb{R}^{d})\text{,
}n\geq1\text{, \quad}\mathsf{S}_{t}^{[n]}:=\frac{1}{B_{n}}\left(
\mathbf{S}_{\left\lfloor tn\right\rfloor }(f)-\frac{\left\lfloor
tn\right\rfloor }{n}A_{n}\right)  \text{,}\label{Eq_DefPSprocesses2}%
\end{equation}
converge to the $\alpha$-stable L\'{e}vy process $\mathsf{S}=(\mathsf{S}%
_{t})_{t\geq0}$ with $\mathsf{S}_{1}\overset{d}{=}S$ in that
\begin{equation}
\mathsf{S}^{[n]}\overset{\nu}{\Longrightarrow}\mathsf{S}\text{ \quad in
}(\mathcal{D}([0,\infty),\mathbb{R}^{d}),\mathcal{J}_{1}%
)\label{Eq_FSLTinThmImpli}%
\end{equation}
for every probability measure $\nu\ll\mu$ on $(X,\mathcal{A})$.
\end{theorem}%

\vspace{0.2cm}%

\begin{remark}
The assumption that $\mu\circ f^{-1}\in\mathrm{DOA}(S)$ entails regular
variation with index $-\alpha$ of $t\mapsto\mu(\left\Vert f\right\Vert >t) $.
Therefore, for any $\eta\in(0,\alpha\wedge1)$, we have $\int\left\Vert
f\right\Vert ^{\eta}\,d\mu<\infty$, and since $\mu(\vartheta_{f}%
>t)=O(\mu(\left\Vert f\right\Vert >t))$ the observable also satisfies the
condition $\int\vartheta_{f}^{\eta}\,d\mu<\infty$ from Proposition
\ref{Prop_GMvsIID}. We do not know whether the latter alone is sufficient for
the conclusion of Theorem \ref{T_SLTgivesFSLT}.
\end{remark}

\begin{remark}
This result remains valid if $\mathsf{S}^{[n]}$ is replaced by $\overline
{\mathsf{S}}^{[n]}=(\overline{\mathsf{S}}_{t}^{[n]})_{t\geq0}$ with
\begin{equation}
\overline{\mathsf{S}}^{[n]}:X\rightarrow\mathcal{D}([0,\infty),\mathbb{R}%
^{d})\text{, }n\geq1\text{, \quad}\overline{\mathsf{S}}_{t}^{[n]}:=\frac
{1}{B_{n}}\left(  \mathbf{S}_{\left\lfloor tn\right\rfloor }(f)-tA_{n}\right)
\text{,}%
\end{equation}
a variant of the partial sum process which some authors prefer (see \cite{T1},
\cite{T2}, \cite{W}). This is clear since $\sup_{t\geq0}\left\Vert
\overline{\mathsf{S}}_{t}^{[n]}-\mathsf{S}_{t}^{[n]}\right\Vert \leq
A_{n}/(nB_{n})\rightarrow0$ as $n\rightarrow\infty$, see
(\ref{Eq_hjjhjhjhjhjhjjjjjjjjjjjjjjjjjjjjjjjjjjjjjjjjjjj2}) in Proposition
\ref{Prop_GMvsIID}.
\end{remark}%

\vspace{0.2cm}%

To facilitate the application of our FSLT, we now provide explicit versions of
two important special cases in Theorems \ref{T_FSL_GM} and \ref{T_FSL_GMddim}
below. The first deals with scalar observables, while the second gives an easy
sufficient condition for a $d$-dimensional observable to lead to an
independent tuple of scalar L\'{e}vy processes in the limit.\newline\newline%
\noindent
\textbf{Special case: }$\mathcal{J}_{1}$-\textbf{FSLT for real-valued
observables.} In the scalar case we obtain%

\vspace{0.2cm}%

\begin{theorem}
[\textbf{Scalar Functional Stable Limit Theorem for GM-maps}]\label{T_FSL_GM}%
Let $(X,\mathcal{A},\mu,T,\xi)$ be a mixing probability preserving
Gibbs-Markov system, and $f:X\rightarrow\mathbb{R}$ an observable in the
domain of attraction of some $\alpha$-stable random variable with $\alpha
\in(0,2)$, as in (\ref{Eq_DomainAttr}). Assume also that
\begin{equation}
\mu(\vartheta_{f}>t)=O(\mu(\left\vert f\right\vert >t))\text{ \quad as
}t\rightarrow\infty\text{.}\label{Eq_ConditionRegularityTail}%
\end{equation}
Then, for the canonical normalizing sequence $(A_{n},B_{n})$ for
$[\alpha,c_{+},c_{-}]$, the partial sum processes $\mathsf{S}^{[n]}%
=(\mathsf{S}_{t}^{[n]})_{t\geq0}$ given by (\ref{Eq_DefPSprocesses2}) converge
to the real-valued $\alpha$-stable L\'{e}vy process $\mathsf{S}=(\mathsf{S}%
_{t})_{t\geq0}$ with $\mathsf{S}_{1}\overset{d}{=}S$, the canonical limit law
for $[\alpha,c_{+},c_{-}]$ characterized by (\ref{Eq_TheFourierTransformOfG}),
so that
\begin{equation}
\mathsf{S}^{[n]}\overset{\nu}{\Longrightarrow}\mathsf{S}\text{ \quad in
}(\mathcal{D}([0,\infty),\mathbb{R}),\mathcal{J}_{1})\label{Eq_FSLTinThm}%
\end{equation}
for every probability measure $\nu\ll\mu$ on $(X,\mathcal{A})$.
\end{theorem}%

\vspace{0.2cm}%

\begin{remark}
\label{Rem_Marta}As mentioned in the introduction, scalar FSLTs for similar
situations have been obtained in \cite{T2}. We briefly outline how our Theorem
\ref{T_FSL_GM} relates to the results of \S 4 in \cite{T2} (in the case of a
mixing probability preserving Gibbs-Markov system $(X,\mathcal{A},\mu,T,\xi)$
and an observable $f:X\rightarrow\mathbb{R}$ in the domain of attraction of
some $\alpha$-stable variable).\newline\textbf{a)} If $f$ is constant on
cylinders, then \cite{T2} shows that it satisfies the $\mathcal{J}_{1}%
$\emph{-}FSLT (\ref{Eq_FSLTinThm}). Theorem \ref{T_FSL_GM} above generalizes
this to all observables which fulfil (\ref{Eq_ConditionRegularityTail}%
).\newline\textbf{b)} Example 1.2 of \cite{T2} shows that
(\ref{Eq_ConditionRegularityTail}) is not necessary for the validity of the
$\mathcal{J}_{1}$\emph{-}FSLT for a \emph{specific} $f$. However, the idea
that the regularity of the particular $f$ used in this example might give a
more general condition sufficient \emph{for all} $f$, which could replace
(\ref{Eq_ConditionRegularityTail}) in Theorem \ref{T_FSL_GM}, is misleading.
Indeed, a different example in \cite{T2} (Example 2.1 there) illustrates the
fact that an observable of the same regularity and tail can fail the
$\mathcal{J}_{1}$\emph{-}FSLT.
\end{remark}%

\noindent
\textbf{Special case: }$\mathcal{J}_{1}$-\textbf{FSLT for tuples with disjoint
tail supports.} Suppose that $f^{(1)},\ldots,f^{(d)}:X\rightarrow\mathbb{R}$
are observables, each attracted to an $\alpha$-stable law, with asymptotically
proportional absolute tails determined on non-overlapping sets, then the
vector-valued observable $f:=(f^{(1)},\ldots,f^{(d)}):X\rightarrow
\mathbb{R}^{d}$ also satisfies a FSLT, the limit being a tuple of independent
$\alpha$-stable L\'{e}vy motions:

\begin{theorem}
[$\mathcal{J}_{1}$\textbf{-convergence to an independent tuple of L\'{e}vy
processes}]\label{T_FSL_GMddim}Let $(X,\mathcal{A},\mu,T,\xi)$ be a mixing
probability preserving Gibbs-Markov system, and let $f^{(1)},\ldots
,f^{(d)}:X\rightarrow\mathbb{R}$ be observables, each satisfying the
assumptions of Theorem \ref{T_FSL_GM}, and such that there are $\alpha
\in(0,2)$, a slowly varying function $\ell$ and, for each $j\in\{1,\ldots,d\}
$, constants $c_{+}^{(j)},c_{-}^{(j)}\geq0$ with $c_{+}^{(j)}+c_{-}^{(j)}>0$
such that, as $t\rightarrow\infty$,%
\begin{equation}
\mu(f^{(j)}>t)\sim c_{+}^{(j)}t^{-\alpha}\ell(t)\text{\quad and}\quad
\mu(f^{(j)}<-t)\sim c_{-}^{(j)}t^{-\alpha}\ell(t)\text{.}%
\label{Eq_AssmIndividTails}%
\end{equation}
Suppose in addition that there is some $M>0$ such that
\begin{equation}
\{\mid f^{(i)}\mid>M\}\cap\{\mid f^{(j)}\mid>M\}=\varnothing\text{ \quad
whenever }i\neq j\text{.}\label{Eq_DisjointTailSupports}%
\end{equation}
Then the partial sum processes $\mathsf{S}^{[n]}=(\mathsf{S}_{t}^{[n]}%
)_{t\geq0}$ of $f:=(f^{(1)},\ldots,f^{(d)}):X\rightarrow\mathbb{R}^{d}$ given
by (\ref{Eq_DefPSprocesses2}) converge to a full $d$-dimensional $\alpha
$-stable L\'{e}vy process $\mathsf{S}=(\mathsf{S}_{t})_{t\geq0}$, so that
\begin{equation}
\mathsf{S}^{[n]}\overset{\nu}{\Longrightarrow}\mathsf{S}\text{ \quad in
}(\mathcal{D}([0,\infty),\mathbb{R}^{d}),\mathcal{J}_{1})
\end{equation}
for every probability measure $\nu\ll\mu$ on $(X,\mathcal{A})$.

Here, $(A_{n}^{(j)},B_{n}^{(j)})_{n\geq1}\subseteq\mathbb{R}\times(0,\infty)$
is the canonical normalizing sequence for $[\alpha,c_{+}^{(j)},c_{-}^{(j)}]$,
and we let $S^{(j)}$ be the corresponding $\alpha$-stable variable as in
(\ref{Eq_TheFourierTransformOfG}), while $B_{n}:=B_{n}^{(1)}$ and
$A_{n}:=(A_{n}^{(1)},\ldots,A_{n}^{(d)})$.

The limit process $\mathsf{S}$ is a tuple $(\mathsf{S}^{(1)},\ldots
,\mathsf{S}^{(d)})$ of $d$ independent scalar $\alpha$-stable L\'{e}vy
processes $(\mathsf{S}_{t}^{(j)})_{t\geq0}$ determined by $\mathsf{S}%
_{1}^{(j)}\overset{d}{=}S^{(j)}$.
\end{theorem}%

\vspace{0.2cm}%

\section{Preparations and Convergence of Marginals}%

\noindent
\textbf{Regularity control.} We next record that due to bounded distortion
(\ref{Eq_BddDistortion}) and the big image property, $a:=\inf_{Z\in\xi}%
\mu(TZ)>0$, a Gibbs-Markov system satisfies
\begin{equation}
\mu(Z\cap T^{-n}E)\leq\frac{e^{R}}{a}\,\mu(Z)\,\mu(E)\text{ \quad}%
\begin{array}
[c]{l}%
\text{whenever }n\geq1\text{, }\\
Z\in\xi_{n}\text{, and }E\in\mathcal{A}\text{.}%
\end{array}
\label{Eq_UpperBoundDistortion}%
\end{equation}
This enables good control of conditional probabilities on cylinders. Let
$(\mathfrak{F},\left\Vert \centerdot\right\Vert )$ be a normed space, and
$f:X\rightarrow\mathfrak{F}$. For our argument it will be crucial to keep
track of the oscillations of ergodic sums on cylinders. We do so using the
functions
\begin{equation}
\vartheta_{f,n}:X\rightarrow\lbrack0,\infty]\text{, \quad}\vartheta
_{f,n}:=\sum_{k=0}^{n-1}\,\theta^{n-k}(\vartheta_{f}\circ T^{k})\text{ \quad
for }n\geq1\text{.}%
\end{equation}
The value which the $\xi_{n}$-measurable function $\vartheta_{f,n}$ takes on
some rank-$n$ cylinder $Z\in\xi_{n}$ will be denoted $\vartheta_{f,n}(Z)$.
This controls the oscillations of $S_{n}(f)$:

\begin{lemma}
[\textbf{Oscillation of ergodic sums on cylinders}]\label{L_OsciOnCyls}Let
$(\mathfrak{F},\left\Vert \centerdot\right\Vert )$ be a normed space,
$(X,\mathcal{A},\mu,T,\xi)$ a probability preserving Gibbs-Markov system, and
$f:X\rightarrow\mathfrak{F}$ an observable with $\vartheta_{f}<\infty$ a.e. on
$X$. Then, for any $n\geq1$ and $Z\in\xi_{n}$, we have
\begin{equation}
\sup\nolimits_{x,y\in Z}\left\Vert \mathbf{S}_{n}(f)(x)-\mathbf{S}%
_{n}(f)(y)\right\Vert \leq\vartheta_{f,n}(Z)\text{.}%
\end{equation}

\end{lemma}%

\vspace{0.1cm}%

\begin{proof}
If $x,y\in Z\in\xi_{n}$ and $k\in\{0,\ldots,n-1\}$, then $\xi(T^{k}%
x)=\xi(T^{k}y)$. Therefore,
\begin{align*}
\left\Vert \mathbf{S}_{n}(f)(x)-\mathbf{S}_{n}(f)(y)\right\Vert  & \leq
\sum_{k=0}^{n-1}\left\Vert f(T^{k}x)-f(T^{k}y)\right\Vert \leq\sum_{k=0}%
^{n-1}D_{\xi(T^{k}x)}(f)\,d_{\theta}(T^{k}x,T^{k}y)\\
& =\sum_{k=0}^{n-1}(\vartheta_{f}\circ T^{k})(x)\,\theta^{n-k}d_{\theta}%
(T^{n}x,T^{n}y)\leq\vartheta_{f,n}(x)\text{,}%
\end{align*}
since $\mathrm{diam}(X)=1$.
\end{proof}%

\vspace{0.2cm}%

This easy estimate will be exploited through the following general observation.

\begin{lemma}
[\textbf{Tails of exponentially weighted ergodic sums}]%
\label{L_TailControlForConvolutionSums}Let $(\mathfrak{F},\left\Vert
\centerdot\right\Vert )$ be a normed space, $T$ a measure-preserving map on
the probability space $(X,\mathcal{A},\mu)$, assume that $g:X\rightarrow
\mathfrak{F}$ is a measurable function, and $\rho\in(0,1)$. Let
\begin{equation}
G_{n}:=\sum_{k=0}^{n-1}\rho^{n-k}\,(g\circ T^{k})\text{ \quad for }%
n\geq1\text{.}%
\end{equation}
If the tail of $\left\Vert g\right\Vert $ satisfies $\mu(\left\Vert
g\right\Vert >t)=O(\tau(t))$ as $t\rightarrow\infty$ for some regularly
varying function $\tau$ of order $-\alpha$ (some $\alpha>0$), then there are
constants $\zeta,s_{\ast}>0$ such that
\begin{equation}
\mu(\left\Vert G_{n}\right\Vert >s)\leq\zeta\,\tau(s)\text{ \quad for }%
n\geq1\text{ and }s\geq s_{\ast}\text{.}\label{Eq_uhgnbuhtnuisjzgfc}%
\end{equation}

\end{lemma}

\begin{proof}
\textbf{(i)} Write $\tau(t)=t^{-\alpha}\ell(t)$, $t>0$, with $\ell$ slowly
varying. Fix some $q\in(\rho,1)$. According to Potter's Theorem for slowly
varying functions (Theorem 1.5.6 of \cite{BGT}), there is some $s_{0}>0$ such
that $\ell(t)/\ell(s)\leq2\max(\left(  t/s\right)  ^{\alpha/2},\left(
t/s\right)  ^{-\alpha/2})$ for $s,t\geq s_{0}$. Therefore there is some
$\kappa>0$ for which
\[
\frac{\ell((1-q)(q/\rho)^{j}s)}{\ell(s)}\leq\kappa\left(  \frac{q}{\rho
}\right)  ^{j\alpha/2}\text{ \quad whenever }j\geq1\text{ and }s\geq
s_{0}\text{.}%
\]
As a consequence,
\begin{equation}%
{\displaystyle\sum\limits_{j=1}^{n}}
\left(  \frac{\rho}{q}\right)  ^{j\alpha}\frac{\ell((1-q)(q/\rho)^{j}s)}%
{\ell(s)}\leq\frac{\kappa}{1-(\rho/q)^{\alpha/2}}\text{ \quad for }s\geq
s_{0}\text{.}\label{Eq_tevcbnxcnbxcnbxcnbcnb}%
\end{equation}
\textbf{(ii)} Now observe, using $T$-invariance of $\mu$, that for $1\leq m<n$
and $t>0$,
\[
\mu\left(
{\displaystyle\sum\limits_{k=0}^{m}}
\rho^{n-k}(\left\Vert g\right\Vert \circ T^{k})>t\right)  \leq\mu(\rho
^{n-m}\left\Vert g\right\Vert >(1-q)t)+\mu\left(
{\displaystyle\sum\limits_{k=0}^{m-1}}
\rho^{n-k}(\left\Vert g\right\Vert \circ T^{k})>qt\right)  \text{.}%
\]
Iterating this we obtain, for every $n\geq1$ and $s>0$,
\begin{align*}
\mu(\left\Vert G_{n}\right\Vert >s)\leq & \mu\left(
{\displaystyle\sum\limits_{k=0}^{n-1}}
\rho^{n-k}(\left\Vert g\right\Vert \circ T^{k})>s\right) \\
\leq & \mu(\rho\left\Vert g\right\Vert >(1-q)s)+\mu\left(
{\displaystyle\sum\limits_{k=0}^{n-2}}
\rho^{n-k}(\left\Vert g\right\Vert \circ T^{k})>qs\right) \\
\vdots & \\
\leq &
{\displaystyle\sum\limits_{j=0}^{n-2}}
\mu(\rho^{j+1}\left\Vert g\right\Vert >(1-q)q^{j}s)+\mu(\rho^{n}\left\Vert
g\right\Vert >q^{n-1}s)\\
\leq &
{\displaystyle\sum\limits_{j=0}^{n-1}}
\mu(\left\Vert g\right\Vert >(1-q)(q/\rho)^{j}s)\text{.}%
\end{align*}
(In the last step we can drop the surplus $\rho$ since $\rho<1$.) With $c>0 $
a constant such that $\mu(\left\Vert g\right\Vert >t)\leq c\tau(t)$ for $t\geq
t_{0}$, we then obtain, recalling (\ref{Eq_tevcbnxcnbxcnbxcnbcnb}),
\begin{align*}
\mu(\left\Vert G_{n}\right\Vert >s)\leq & \tau(s)\,c%
{\displaystyle\sum\limits_{j=1}^{n}}
\frac{\tau((1-q)(q/\rho)^{j}s)}{\tau(s)}\\
=  & \tau(s)\,c(1-q)^{-\alpha}%
{\displaystyle\sum\limits_{j=1}^{n}}
\left(  \frac{\rho}{q}\right)  ^{j\alpha}\frac{\ell((1-q)(q/\rho)^{j}s)}%
{\ell(s)}\\
\leq & \tau(s)\,\frac{c\kappa(1-q)^{-\alpha}}{1-(\rho/q)^{\alpha/2}}\text{
\quad for }s\geq\max\left(  s_{0},\frac{\rho\,t_{0}}{(1-q)q}\right)  \text{,}%
\end{align*}
which establishes our claim (\ref{Eq_uhgnbuhtnuisjzgfc}).
\end{proof}%

\vspace{0.2cm}%

This leads to

\begin{lemma}
[\textbf{Uniform tail estimate for the} $\vartheta_{f,k}$]\label{L_ThetaFTail}%
Under the assumptions of Theorem \ref{T_SLTgivesFSLT}, there is some constant
$\zeta_{f}>0$ such that
\begin{equation}
\underset{n\rightarrow\infty}{\overline{\lim}}\,\,n\,\sup_{k\geq1}\mu\left(
\frac{\vartheta_{f,k}}{B_{n}}>\varepsilon\right)  \leq\zeta_{f}\,\varepsilon
^{-\alpha}\text{\quad for }\varepsilon>0\text{.}%
\label{Eq_sacdsadsadsadsadsafdsafsa}%
\end{equation}

\end{lemma}

\begin{proof}
Take any $i\in\{1,\ldots,d\},$ and consider the tail $\tau_{f^{(i)}}$ of
$\left\vert f^{(i)}\right\vert $. Then regular variation of $(B_{n})$ and
(\ref{Eq_shvfbjsfdbvjhdbbbbbbyyyyyyyyyyy2}) of Proposition \ref{Prop_GMvsIID}
show that for any $\varepsilon>0$ there is some $n_{0}^{(i)}(\varepsilon)$
such that
\begin{equation}
\tau_{f^{(i)}}\left(  \frac{\varepsilon}{d}\,B_{n}\right)  \leq2c^{(i)}%
d^{\alpha}\,\frac{\varepsilon^{-\alpha}}{n}\text{ \quad for }n\geq n_{0}%
^{(i)}(\varepsilon)\text{.}\label{Eq_vbvbvbvbvbvbbbbbbbbbb}%
\end{equation}
\newline Our regularity assumption allows us to apply Lemma
\ref{L_TailControlForConvolutionSums} to $g:=\vartheta_{f^{(i)}}$ and
$\tau:=\tau_{f^{(i)}}$ to obtain $\zeta^{(i)},s_{\ast}^{(i)}>0$ such that for
all $\varepsilon>0$ we have $\mu\left(  \vartheta_{f^{(i)},k}/B_{n}%
>\varepsilon\right)  \leq\zeta^{(i)}\,\tau_{f^{(i)}}\left(  \varepsilon
\,B_{n}\right)  $ whenever $k\geq1$ and $n\geq n_{1}^{(i)}(\varepsilon)$ (so
large that $\varepsilon\,B_{n}\geq s_{\ast}^{(i)}$). Now $D_{Z}(f)\leq
\sum_{i=1}^{d}D_{Z}(f^{(i)})$, and hence $\vartheta_{f,k}\leq\sum_{i=1}%
^{d}\vartheta_{f^{(i)},k}$, so that for any $\varepsilon>0$,
\[
\mu\left(  \frac{\vartheta_{f,k}}{B_{n}}>\varepsilon\right)  \leq\sum
_{i=1}^{d}\mu\left(  \frac{\vartheta_{f^{(i)},k}}{B_{n}}>\frac{\varepsilon}%
{d}\right)  \leq\sum_{i=1}^{d}\zeta^{(i)}\,\tau_{f^{(i)}}\left(
\frac{\varepsilon}{d}\,B_{n}\right)  \text{ \quad}%
\begin{array}
[c]{c}%
\text{for }k\geq1\text{ and }\\
n\geq\max\limits_{1\leq i\leq d}n_{1}^{(i)}(\varepsilon)\text{.}%
\end{array}
\]
Combined with (\ref{Eq_vbvbvbvbvbvbbbbbbbbbb}) this gives
(\ref{Eq_sacdsadsadsadsadsafdsafsa}) with $\zeta_{f}:=2d^{\alpha}\sum
_{i=1}^{d}\zeta^{(i)}c^{(i)}$.
\end{proof}%

\vspace{0.2cm}%
%

\noindent
\textbf{Uniform control under conditional measures.} By standard arguments,
the good distortion properties of $T$ can also be expressed in terms of the
\emph{transfer operator} $\widehat{T}:L_{1}(\mu)\rightarrow L_{1}(\mu)$ of
$T$, which is characterized by $\int f\circ T\cdot u\,d\mu=\int f\cdot
\widehat{T}u\,d\mu$ for $f\in L_{\infty}(\mu)$ and $u\in L_{1}(\mu)$. We are
going to use this via

\begin{lemma}
[\textbf{An }$L_{1}$\textbf{-compact invariant set for }$\widehat{T}$%
]\label{L_CptEmbedding}Let $(X,\mathcal{A},\mu,T,\xi)$ be a probability
preserving Gibbs-Markov map. There is some strongly compact convex set
$\mathfrak{H}\subseteq L_{1}(\mu)$ such that $\widehat{T}\mathfrak{H}%
\subseteq\mathfrak{H}$, while for every $n\geq1$ and $W\in\xi_{n}$, the
normalized density $\mu(W)^{-1}\widehat{T}^{n}1_{W}$ belongs to $\mathfrak{H}$.
\end{lemma}

\begin{proof}
This follows from bounded distortion, see for example \cite{AD}. One can
choose $\mathfrak{H}:=\{f\in L_{1}(\mu):f\geq0$, $\int f\,d\mu=1$, and $f$ has
a version with $D_{\xi}(f)<K\}$ for a suitable constant $K>0$.
\end{proof}%

\vspace{0.2cm}%

We provide one more abstract lemma which is useful for proving convergence of
finite-dimensional marginals.

\begin{lemma}
[\textbf{Uniform changes of measures}]\label{L_ChangeMeas}Let $(X,\mathcal{A}%
,\mu,T)$ be an ergodic probability preserving system, and $(G_{n})_{n\geq0}$ a
uniformly bounded sequence of measurable functions $G_{n}:X\rightarrow
\lbrack0,\infty)$ satisfying
\begin{equation}
G_{n}\circ T-G_{n}\overset{\mu}{\longrightarrow}0\text{.}%
\label{Eq_AsyTInvarGn}%
\end{equation}
Suppose that $\mathfrak{H}$ is a family of probability densities, strongly
compact in $L_{1}(\mu)$. Then
\begin{equation}
\int_{X}G_{n}\cdot v\,d\mu-\int_{X}G_{n}\cdot v^{\ast}\,d\mu\longrightarrow
0\text{ \quad}%
\begin{array}
[c]{l}%
\text{as }n\rightarrow\infty\text{,}\\
\text{uniformly in }u,u^{\ast}\in\mathfrak{H}\text{.}%
\end{array}
\end{equation}

\end{lemma}

\begin{proof}
This follows from a classical companion (see \cite{Y} or Theorem 2 of
\cite{Z7}) to the mean ergodic theorem. In fact, it is contained in
Proposition 3.1 of \cite{Z6}.
\end{proof}%

\vspace{0.2cm}%
%

\noindent
\textbf{Convergence of marginals.} We can then establish

\begin{proposition}
[\textbf{Convergence of finite-dimensional marginals}]\label{P_Finidim}Under
the assumptions of Theorem \ref{T_SLTgivesFSLT} we have, for all $m\geq1$,
\begin{equation}
(\mathsf{S}_{t_{1}}^{[n]},\ldots,\mathsf{S}_{t_{m}}^{[n]})\overset{\mu
}{\Longrightarrow}(\mathsf{S}_{t_{1}},\ldots,\mathsf{S}_{t_{m}})\text{ \quad
as }n\rightarrow\infty\text{, for }0\leq t_{1}<\ldots<t_{m}\leq1\text{.
}\label{Eq_FiniDim}%
\end{equation}

\end{proposition}

\begin{proof}
\textbf{(i)} Asumption (\ref{Eq_BasicAssmThmImpli}) implies that
(\ref{Eq_FiniDim}) is satisfied for $m=1$. For the inductive step, we fix any
$m\geq1$ and assume validity of (\ref{Eq_FiniDim}). To prove (\ref{Eq_FiniDim}%
) with $m$ replaced by $m+1$ we fix any tuple $0<t<t_{1}<\ldots<t_{m}\leq1$
(the case $t=0$ being trivial), and any $s=(s^{(1)},\ldots,s^{(d)}%
)\in\mathbb{R}^{d}$ for which $\Pr[\mathsf{S}_{t}\leq s]>0$, where
$(r^{(1)},\ldots,r^{(d)})\leq(s^{(1)},\ldots,s^{(d)})$ means that $r^{(i)}\leq
s^{(i)}$ for all $i\in\{1,\ldots,d\}$. We are going to show that
\begin{equation}
(\mathsf{S}_{t_{1}}^{[n]}-\mathsf{S}_{t}^{[n]},\ldots,\mathsf{S}_{t_{m}}%
^{[n]}-\mathsf{S}_{t}^{[n]})\overset{\mu_{E_{n}}}{\Longrightarrow}%
(\mathsf{S}_{t_{1}-t},\ldots,\mathsf{S}_{t_{m}-t})\text{ \quad as
}n\rightarrow\infty\text{,}\label{Eq_cbxvuzserfusrzve}%
\end{equation}
where $E_{n}:=\{\mathsf{S}_{t}^{[n]}\leq s\}$. This suffices since
$\mathsf{S}$ is a stable L\'{e}vy motion, and the $m=1$ case of
(\ref{Eq_FiniDim}) guarantees that $\mu(E_{n})\rightarrow\Pr[\mathsf{S}%
_{t}\leq s]>0$.\newline\newline\textbf{(ii)} We will work with conditioning
events $F_{n}$ more convenient than the $E_{n}$. Define $E_{n}^{\prime
}:=\{\mathsf{S}_{t}^{[n]}-B_{n}^{-1}\vartheta_{f,\left\lfloor tn\right\rfloor
}(1,\ldots,1)\leq s\}$\ and $F_{n}:=%
{\textstyle\bigcup\nolimits_{Z\in\xi_{\left\lfloor tn\right\rfloor
}:\mathsf{S}_{t}^{[n]}\leq s\text{ somewhere on }Z}}
Z$, $n\geq1$. Obviously, $E_{n}\subseteq F_{n}$. Now $\vartheta_{f^{(i)}%
,j}\leq\vartheta_{f,j}$ if $f=(f^{(1)},\ldots,f^{(d)})$, and Lemma
\ref{L_OsciOnCyls} shows that $F_{n}\subseteq E_{n}^{\prime}$. Next, Lemma
\ref{L_ThetaFTail} gives
\begin{equation}
\vartheta_{f,\left\lfloor tn\right\rfloor }/B_{n}\overset{\mu}{\longrightarrow
}0\text{\quad as }n\rightarrow\infty\text{.}%
\end{equation}
Since $\mathsf{S}_{t}$ has a continuous distribution, $(\mu(E_{n}^{\prime}))$
has the same limit as $(\mu(E_{n}))$, and hence we also have $\mu
(F_{n})\rightarrow\Pr[\mathsf{S}_{t}\leq s]$. Therefore,
(\ref{Eq_cbxvuzserfusrzve}) is equivalent to
\begin{equation}
(\mathsf{S}_{t_{1}}^{[n]}-\mathsf{S}_{t}^{[n]},\ldots,\mathsf{S}_{t_{m}}%
^{[n]}-\mathsf{S}_{t}^{[n]})\overset{\mu_{F_{n}}}{\Longrightarrow}%
(\mathsf{S}_{t_{1}-t},\ldots,\mathsf{S}_{t_{m}-t})\text{ \quad as
}n\rightarrow\infty\text{.}\label{Eq_vdbhcdhbcdhbhbcch}%
\end{equation}
But $\mathsf{S}_{t^{\prime\prime}}^{[n]}-\mathsf{S}_{t^{\prime}}^{[n]}%
=B_{n}^{-1}(\mathbf{S}_{\left\lfloor t^{\prime\prime}n\right\rfloor
-\left\lfloor t^{\prime}n\right\rfloor }(f)-(\left\lfloor t^{\prime\prime
}n\right\rfloor -\left\lfloor t^{\prime}n\right\rfloor )A_{n}/n)\circ
T^{\left\lfloor t^{\prime}n\right\rfloor }$ for any $t^{\prime}<t^{\prime
\prime}$. Also, since $\left\vert \left\lfloor t^{\prime\prime}n\right\rfloor
-\left\lfloor t^{\prime}n\right\rfloor -\left\lfloor (t^{\prime\prime
}-t^{\prime})n\right\rfloor \right\vert \leq2$, we get
\[
\left\Vert (\mathsf{S}_{t^{\prime\prime}}^{[n]}-\mathsf{S}_{t^{\prime}}%
^{[n]})-\mathsf{S}_{t^{\prime\prime}-t^{\prime\prime}}^{[n]}\circ
T^{\left\lfloor t^{\prime}n\right\rfloor }\right\Vert \overset{\mu
}{\longrightarrow}0\text{\quad as }n\rightarrow\infty\text{.}%
\]
(Use (\ref{Eq_hjjhjhjhjhjhjjjjjjjjjjjjjjjjjjjjjjjjjjjjjjjjjjj2}) and the fact
that $(f\circ T^{k})/B_{k}\overset{\mu}{\longrightarrow}0$, which is immediate
from (\ref{Eq_DomainAttr}) and (\ref{Eq_DefBn}).) Therefore
(\ref{Eq_vdbhcdhbcdhbhbcch}) is equivalent to%
\begin{equation}
(\mathsf{S}_{t_{1}-t}^{[n]},\ldots,\mathsf{S}_{t_{m}-t}^{[n]})\overset
{\mu_{F_{n}}\circ T^{-\left\lfloor tn\right\rfloor }}{\Longrightarrow
}(\mathsf{S}_{t_{1}-t},\ldots,\mathsf{S}_{t_{m}-t})\text{ \quad as
}n\rightarrow\infty\text{.}\label{Eq_jnjhjnbjnjjjjjjnnj}%
\end{equation}
Note that $F_{n}$ is a $\xi_{\left\lfloor tn\right\rfloor }$-measurable set.
Therefore the density of $\mu_{F_{n}}\circ T^{-\left\lfloor tn\right\rfloor }%
$, that is, $u_{n}:=\widehat{T}^{\left\lfloor tn\right\rfloor }(\mu
(F_{n})^{-1}1_{F_{n}})$ belongs to the closed convex set $\mathfrak{H}$ of
Lemma \ref{L_CptEmbedding}.\newline\newline\textbf{(iii)} The desired
convergence (\ref{Eq_jnjhjnbjnjjjjjjnnj}) can be established by checking that
for every $G:\mathbb{R}^{md}\rightarrow\mathbb{R}$ of the form $G(s_{1}%
,\ldots,s_{m})=g_{1}(s_{1})\cdots g_{m}(s_{m})$ with bounded Lipschitz
functions $g_{j}:\mathbb{R}^{d}\rightarrow\mathbb{R}$ we have
\begin{equation}
\int G(\mathsf{S}_{t_{1}-t}^{[n]},\ldots,\mathsf{S}_{t_{m}-t}^{[n]}%
)\,u_{n}\,d\mu\longrightarrow\mathbb{E}[G(\mathsf{S}_{t_{1}-t},\ldots
,\mathsf{S}_{t_{m}-t})]\text{\quad as }n\rightarrow\infty\text{.}%
\label{Eq_easxeww}%
\end{equation}
Due to assumption (\ref{Eq_FiniDim}), (\ref{Eq_jnjhjnbjnjjjjjjnnj}) is valid
if the $\mu_{F_{n}}\circ T^{-\left\lfloor tn\right\rfloor }$ are replaced by
the single measure $\mu$, and hence (\ref{Eq_easxeww}) is valid if the $u_{n}$
are replaced by the density $1_{X}$ of $\mu$. Therefore, (\ref{Eq_easxeww})
follows once we prove that
\begin{equation}
\int G(\mathsf{S}_{t_{1}-t}^{[n]},\ldots,\mathsf{S}_{t_{m}-t}^{[n]}%
)\,(u_{n}-1_{X})\,d\mu\longrightarrow0\text{\quad as }n\rightarrow
\infty\text{.}\label{Eq_wbsbwbwbw}%
\end{equation}
But letting $G_{n}:=G(\mathsf{S}_{t_{1}-t}^{[n]},\ldots,\mathsf{S}_{t_{m}%
-t}^{[n]})$ it is easy to see that
\[
\left\vert G_{n}\circ T-G_{n}\right\vert \leq\frac{L\Gamma^{m-1}}{B_{n}}%
\sum_{j=1}^{m}\left(  \left\Vert f\right\Vert +\left\Vert f\right\Vert \circ
T^{\left\lfloor (t_{j}-t)n\right\rfloor }\right)  \text{,}%
\]
with $L$ a common Lipschitz constant for the $g_{j}$, and $\Gamma:=\max_{1\leq
j\leq m}\sup\left\Vert g_{j}\right\Vert $. Since $T$ preserves $\mu$ and
$B_{n}\rightarrow\infty$, the asymptotic invariance property
(\ref{Eq_AsyTInvarGn}) follows. Lemma \ref{L_ChangeMeas} then gives
(\ref{Eq_wbsbwbwbw}) since $u_{n}\in\mathfrak{H}$ for all $n\geq1$.
\end{proof}%

\vspace{0.2cm}%

\section{Maximal Inequalities and Tightness}%

\noindent
\textbf{Maximal inequalities.} The proof of tightness will depend on the
following maximal inequalities, which constitute the main technical tool of
the present paper.\ (An easier version of these arguments has been used in
\cite{ZStochDyn}.) Let $(\mathfrak{F},\left\Vert \centerdot\right\Vert )$ be a
normed space.

\begin{lemma}
[\textbf{Maximal inequalities for ergodic sums}]\label{L_MaxIneq1}Let
$(X,\mathcal{A},\mu,T,\xi)$ be a probability preserving Gibbs-Markov map, and
$g:X\rightarrow\mathfrak{F}$ an observable with $\vartheta_{g}<\infty$ a.e. on
$X$. Denote $S_{n}=\mathbf{S}_{n}(g)$, $n\geq0$. Then, for any $n\geq1$ and
$\kappa>0$,%
\begin{equation}
\mu\left(  \max_{1\leq k\leq n}\left\Vert S_{k}\right\Vert >\kappa\right)
\leq\frac{2e^{R}}{a}\,\max_{1\leq k\leq n}\mu\left(  \left\Vert S_{k}%
\right\Vert >\frac{\kappa}{4}\right)  +n\max_{1\leq k\leq n}\mu\left(
\vartheta_{g,k}>\frac{\kappa}{4}\right)  \text{,}\label{Eq_MI1_a}%
\end{equation}
while
\begin{equation}
\mu\left(  \max_{1\leq k\leq n}\left\Vert S_{n}-S_{k}\right\Vert
>\kappa\right)  \leq\mu\left(  \max_{1\leq k\leq n}\left\Vert S_{k}\right\Vert
>\frac{\kappa}{2}\right)  \text{.}\label{Eq_MI1_b}%
\end{equation}
Moreover, for any $n\geq1$ and $\kappa>0$,
\begin{align}
&  \mu\left(  \max_{1\leq i<j<l\leq n}\left(  \left\Vert S_{j}-S_{i}%
\right\Vert \wedge\left\Vert S_{l}-S_{j}\right\Vert \right)  >\kappa\right)
\nonumber\\
&  \leq\frac{e^{R}}{a}\mu\left(  \max_{1\leq k\leq n}\left\Vert S_{k}%
\right\Vert >\frac{\kappa}{4}\right)  \,\left[  \mu\left(  \max_{1\leq k\leq
n}\left\Vert S_{k}\right\Vert >\frac{\kappa}{4}\right)  +n\max_{1\leq k\leq
n}\mu\left(  \vartheta_{g,k}>\frac{\kappa}{4}\right)  \right]  \text{.}%
\label{Eq_MI1_ccc}%
\end{align}

\end{lemma}

\begin{proof}
\textbf{(i)} We fix $\kappa>0$ and define families of cylinders by
\begin{align}
\gamma_{1}(\kappa) &  :=\left\{  Z\in\xi_{1}:\sup_{Z}\left\Vert S_{1}%
\right\Vert >\kappa\right\}  \text{ \quad and, for }k\geq1\text{,
}\label{Eq_DefGammas}\\
\gamma_{k+1}(\kappa) &  :=\left\{  Z\in\left(
{\textstyle\bigcup\nolimits_{j=1}^{k}}
{\textstyle\bigcup\nolimits_{W\in\gamma_{j}(\kappa)}}
W\right)  ^{c}\cap\xi_{k+1}:\sup_{Z}\left\Vert S_{k+1}\right\Vert
>\kappa\right\} \nonumber
\end{align}
(where, for $E$ measurable w.r.t. a partition $\eta$, $E\cap\eta:=\{T\in
\eta:T\subseteq E\}$). Then,
\begin{equation}
\left\{  \max_{1\leq k<n}\left\Vert S_{k}\right\Vert >\kappa\right\}
\subseteq\bigcup_{k=1}^{n-1}\bigcup_{Z\in\gamma_{k}(\kappa)}Z\text{
\quad(disjoint),}\label{Eq_uauauauauauauaua}%
\end{equation}
and, since $\inf_{Z}\left\Vert S_{k}\right\Vert >\kappa-\vartheta_{g,k}(Z)$
for $Z\in\gamma_{k}(\kappa)$ by Lemma \ref{L_OsciOnCyls}, we get
\begin{multline*}
\left\{  \max_{1\leq k<n}\left\Vert S_{k}\right\Vert >\kappa\right\}
\cap\left\{  \left\Vert S_{n}\right\Vert \leq\frac{\kappa}{2}\right\}
\subseteq\bigcup_{k=1}^{n-1}\bigcup_{Z\in\gamma_{k}(\kappa)}Z\cap\left\{
\left\Vert S_{n}\right\Vert \leq\frac{\kappa}{2}\right\} \\
\subseteq\bigcup_{k=1}^{n-1}\bigcup_{Z\in\gamma_{k}(\kappa)}Z\cap\left(
\left\{  \left\Vert S_{n}-S_{k}\right\Vert >\frac{\kappa}{4}\right\}
\cup\left\{  \vartheta_{g,k}>\frac{\kappa}{4}\right\}  \right)  \text{.}%
\end{multline*}
According to (\ref{Eq_UpperBoundDistortion}) we see that for any $Z\in
\gamma_{k}(\kappa)$, $1\leq k<n$,
\begin{align*}
\mu\left(  Z\cap\left\{  \left\Vert S_{n}-S_{k}\right\Vert >\frac{\kappa}%
{4}\right\}  \right)   & =\mu\left(  Z\cap T^{-k}\left\{  \left\Vert
S_{n-k}\right\Vert >\frac{\kappa}{4}\right\}  \right) \\
& \leq\frac{e^{R}}{a}\,\mu(Z)\,\mu\left(  \left\Vert S_{n-k}\right\Vert
>\frac{\kappa}{4}\right)  \text{.}%
\end{align*}
Combining these observations we find that
\begin{multline*}
\mu\left(  \left\{  \max_{1\leq k<n}\left\Vert S_{k}\right\Vert >\kappa
\right\}  \cap\left\{  \left\Vert S_{n}\right\Vert \leq\frac{\kappa}%
{2}\right\}  \right) \\
\leq\frac{e^{R}}{a}\sum_{k=1}^{n-1}\sum_{Z\in\gamma_{k}(\kappa)}\mu
(Z)\,\mu\left(  \left\Vert S_{n-k}\right\Vert >\frac{\kappa}{4}\right)
+\sum_{k=1}^{n-1}\sum_{Z\in\gamma_{k}(\kappa)}\mu\left(  Z\cap\left\{
\vartheta_{g,k}>\frac{\kappa}{4}\right\}  \right) \\
\leq\frac{e^{R}}{a}\,\max_{1\leq k<n}\mu\left(  \left\Vert S_{k}\right\Vert
>\frac{\kappa}{4}\right)  +n\max_{1\leq k<n}\mu\left(  \vartheta_{g,k}%
>\frac{\kappa}{4}\right)  \text{,}%
\end{multline*}
where the last step again uses that $\{Z\in\gamma_{k}(\kappa):k\geq1\}$ is a
family of pairwise disjoint sets. This implies our maximal inequality
(\ref{Eq_MI1_a}) if we also note that
\[
\mu\left(  \max_{1\leq k\leq n}\left\Vert S_{k}\right\Vert >\kappa\right)
\leq\mu\left(  \left\Vert S_{n}\right\Vert >\frac{\kappa}{2}\right)
+\mu\left(  \left\{  \max_{1\leq k<n}\left\Vert S_{k}\right\Vert
>\kappa\right\}  \cap\left\{  \left\Vert S_{n}\right\Vert \leq\frac{\kappa}%
{2}\right\}  \right)  \text{.}%
\]
The second inequality (\ref{Eq_MI1_b}) of our lemma is immediate from
\[
\left\{  \max_{1\leq k\leq n}\left\Vert S_{n}-S_{k}\right\Vert >\kappa
\right\}  \subseteq\left\{  \left\Vert S_{n}\right\Vert >\frac{\kappa}%
{2}\right\}  \cup\left\{  \max_{1\leq k<n}\left\Vert S_{k}\right\Vert
>\frac{\kappa}{2}\right\}  \text{.}%
\]
%

\noindent
\textbf{(ii)} Fixing $\kappa>0$, we first note that by
(\ref{Eq_uauauauauauauaua}),
\[
\left\{  \max_{1\leq i<j<n}\left\Vert S_{j}-S_{i}\right\Vert >\kappa\right\}
\subseteq\left\{  \max_{1\leq k<n}\left\Vert S_{k}\right\Vert >\frac{\kappa
}{2}\right\}  \subseteq\bigcup_{k=1}^{n-1}\bigcup_{Z\in\gamma_{k}(\kappa
/2)}Z\text{ \quad(disjoint).}%
\]
Assume that $x\in\{\max_{1\leq i<j<l\leq n}(\left\Vert S_{j}-S_{i}\right\Vert
\wedge\left\Vert S_{l}-S_{j}\right\Vert )>\kappa\}$. Then $\left\Vert
S_{m}(x)\right\Vert >\kappa/2$ for some $m\leq n$, and therefore there are
$k\in\{1,\ldots,n\}$ and $Z\in\gamma_{k}(\kappa/2)$ such that $x\in Z$ (use
(\ref{Eq_uauauauauauauaua}) again). Choose $1\leq i<j<l\leq n$ such that
$\left\Vert S_{j}(x)-S_{i}(x)\right\Vert >\kappa$ and $\left\Vert
S_{l}(x)-S_{j}(x)\right\Vert >\kappa$. According to the definition of
$\gamma_{k}(\kappa/2)$ above, we have $\left\Vert S_{h}(x)\right\Vert
\leq\kappa/2$ for all $h<k$, so that (due to $\left\Vert S_{j}(x)-S_{i}%
(x)\right\Vert >\kappa$) necessarily $j\geq k$. We claim that
\[
\max_{k<h\leq n}\left\Vert S_{h}(x)-S_{k}(x)\right\Vert >\kappa/2\text{.}%
\]
In case $j=k$ this is clear for $h=l$, by our choice of $j$ and $l$. On the
other hand, if $j>k$, then $\left\Vert S_{l}(x)-S_{j}(x)\right\Vert >\kappa$
ensures that $\left\Vert S_{j}(x)-S_{k}(x)\right\Vert >\kappa/2$ or
$\left\Vert S_{l}(x)-S_{k}(x)\right\Vert >\kappa/2$, and we can take $h=j$ or
$h=l$, proving our claim.

We therefore see that for any $k\in\{1,\ldots,n\}$ and $Z\in\gamma_{k}%
(\kappa/2)$,
\begin{align*}
Z\cap\left\{  \max_{1\leq i<j<l\leq n}\left(  \left\Vert S_{j}-S_{i}%
\right\Vert \wedge\left\Vert S_{l}-S_{j}\right\Vert \right)  >\kappa\right\}
& \subseteq Z\cap\left\{  \max_{k<h\leq n}\left\Vert S_{h}-S_{k}\right\Vert
>\frac{\kappa}{2}\right\} \\
& =Z\cap T^{-k}\left\{  \max_{1\leq l\leq n-k}\left\Vert S_{l}\right\Vert
>\frac{\kappa}{2}\right\}  \text{,}%
\end{align*}
and hence
\begin{multline*}
\mu\left(  \max_{1\leq i<j<l\leq n}\left(  \left\Vert S_{j}-S_{i}\right\Vert
\wedge\left\Vert S_{l}-S_{j}\right\Vert \right)  >\kappa\right) \\
\leq\sum_{k=1}^{n}\sum_{Z\in\gamma_{k}(\kappa/2)}\mu\left(  Z\cap
T^{-k}\left\{  \max_{1\leq l\leq n-k}\left\Vert S_{l}\right\Vert >\frac
{\kappa}{2}\right\}  \right) \\
\leq\frac{e^{R}}{a}\sum_{k=1}^{n}\sum_{Z\in\gamma_{k}(\kappa/2)}\mu
(Z)\,\mu\left(  \max_{1\leq l\leq n-k}\left\Vert S_{l}\right\Vert
>\frac{\kappa}{2}\right) \\
\leq\frac{e^{R}}{a}\,\mu\left(  \max_{1\leq l\leq n}\left\Vert S_{l}%
\right\Vert >\frac{\kappa}{2}\right)  \,\mu\left(  \bigcup_{k=1}^{n}%
\bigcup_{Z\in\gamma_{k}(\kappa/2)}Z\right)  \text{.}%
\end{multline*}
But as $\inf_{Z}\left\Vert S_{k}\right\Vert >\kappa/2-\vartheta_{g,k}(Z)$ for
$Z\in\gamma_{k}(\kappa/2)$, we have $Z\subseteq\{\left\Vert S_{k}\right\Vert
>\kappa/4\}\cup\{\vartheta_{g,k}>\kappa/4\}$ for such $Z $, and therefore
\[
\mu\left(  \bigcup_{k=1}^{n}\bigcup_{Z\in\gamma_{k}(\kappa/2)}Z\right)
\leq\,\mu\left(  \max_{1\leq k\leq n}\left\Vert S_{k}\right\Vert >\frac
{\kappa}{4}\right)  +n\max_{1\leq k\leq n}\mu\left(  \vartheta_{g,k}%
>\frac{\kappa}{4}\right)  \text{.}%
\]
Combining the last two estimates yields inequality (\ref{Eq_MI1_ccc}).
\end{proof}%

\vspace{0.2cm}%
%

\noindent
\textbf{Control of translation sequence.} We shall also use the following
observation regarding the sequence $(A_{n})$. Note that the assertion of the
lemma below is trivial in many cases (when $A_{n}=0$ or $A_{n}=const\cdot n$).
The $d$-dimensional $\alpha$-stable L\'{e}vy process $\mathsf{S}%
=(\mathsf{S}_{t})_{t\geq0}$ with $\mathsf{S}_{1}\overset{d}{=}S$ satisfies
$\mathsf{S}_{t}\overset{d}{=}t^{1/\alpha}S+\mathsf{a}_{t}$, $t\geq0$, for some
continuous function $t\mapsto\mathsf{a}_{t}=(\mathsf{a}_{t}^{(1)}%
,\ldots,\mathsf{a}_{t}^{(d)})$, $t\geq0$, which appears in the following statement.

\begin{lemma}
[\textbf{Uniform control of translation vectors} $A_{n}$]\label{L_TheAn}Under
the assumptions of Theorem \ref{T_SLTgivesFSLT}, the normalizing sequence
$(A_{n},B_{n})$ satisfies
\begin{equation}
\max_{1\leq k\leq n}\left\Vert \frac{1}{B_{n}}\left(  A_{k}-\frac{k}{n}%
A_{n}\right)  -\mathsf{a}_{\frac{k}{n}}\right\Vert \longrightarrow0\text{
\quad as }n\rightarrow\infty\text{.}\label{Eq_UnifControlAn}%
\end{equation}

\end{lemma}

\begin{proof}
\textbf{(i)} Assume the contrary, and write $A_{n}=(A_{n}^{(1)},\ldots
,A_{n}^{(d)})$. Then there are $i\in\{1,\ldots,d\}$, $\eta>0$ and sequences
$(n_{j})_{j\geq1}$ and $(k_{j})_{j\geq1}$ in $\mathbb{N}$ such that $k_{j}\leq
n_{j}$ and
\begin{equation}
\left\vert \frac{1}{B_{n_{j}}}\left(  A_{k_{j}}^{(i)}-\frac{k_{j}}{n_{j}%
}A_{n_{j}}^{(i)}\right)  -\mathsf{a}_{\frac{k_{j}}{n_{j}}}^{(i)}\right\vert
>\eta\text{ \quad for }j\geq1\text{,}\label{Eq_TheBadGuys}%
\end{equation}
while $s_{j}:=k_{j}/n_{j}\longrightarrow s$\ as $j\rightarrow\infty$ for some
$s\in\lbrack0,1]$. But we are going to show that for any $s$ and
$(s_{n})_{n\geq1}$ in $(0,1]$ with $s_{n}\rightarrow s$,
\begin{equation}
\frac{1}{B_{n}}\left(  A_{\left\lfloor s_{n}n\right\rfloor }^{(i)}-s_{n}%
A_{n}^{(i)}\right)  -\mathsf{a}_{s}^{(i)}\longrightarrow0\text{ \quad as
}n\rightarrow\infty\text{,}\label{Eq_TheGoalNow}%
\end{equation}
which contradicts (\ref{Eq_TheBadGuys}) and therefore proves
(\ref{Eq_UnifControlAn}).\newline\newline\textbf{(ii)} To validate
(\ref{Eq_TheGoalNow}), recall first that we have observed (Proposition
\ref{Prop_GMvsIID}) that if $(Z_{k})_{k\geq1}$ is an iid sequence of random
variables on some probability space $(\Omega,\mathcal{F},\mathrm{\Pr})$ with
the same distribution as $f^{(i)}$, and $\widehat{S}_{n}:=\sum_{k=0}%
^{n-1}Z_{k}$, then $B_{n}^{-1}(\widehat{S}_{n}-A_{n}^{(i)})\Longrightarrow
S^{(i)}$.

Define $\widehat{\mathsf{S}}^{[n]}=(\widehat{\mathsf{S}}_{t}^{[n]}%
)_{t\in\lbrack0,1]}$ by $\widehat{\mathsf{S}}_{t}^{[n]}:=B_{n}^{-1}%
(\widehat{S}_{\left\lfloor tn\right\rfloor }-\left\lfloor tn\right\rfloor
A_{n}^{(i)}/n)$. Now use Skorokhod's classical $\mathcal{J}_{1}$-FSLT
(\cite{Sk2}) to see that $\widehat{\mathsf{S}}^{[n]}\Longrightarrow
\widehat{\mathsf{S}}$ in the $\mathcal{J}_{1}$-topology, with $\widehat
{\mathsf{S}}$ the one-dimensional stable L\'{e}vy process with $\widehat
{\mathsf{S}}_{1}\overset{d}{=}S^{(i)}$.

In particular, $\widehat{\mathsf{S}}_{1}^{[s_{n}n]}\Longrightarrow
\widehat{\mathsf{S}}_{1}$, and due to regular variation of $(B_{n})$,
\begin{equation}
V_{n}:=\frac{B_{\left\lfloor s_{n}n\right\rfloor }}{B_{n}}\frac{1}%
{B_{\left\lfloor s_{n}n\right\rfloor }}(\widehat{S}_{\left\lfloor
s_{n}n\right\rfloor }-A_{s_{n}n}^{(i)})\Longrightarrow s^{1/\alpha}%
\widehat{\mathsf{S}}_{1}\text{ \quad as }n\rightarrow\infty\text{.}%
\label{Eq_FirstObs}%
\end{equation}
Now consider homeomorphisms $\lambda_{n}\in\Lambda$ of $[0,1]$ with
$\lambda_{n}(s)=s_{n}$, and affine on $[0,s]$ and $[s,1]$, then $\sup
_{[0,1]}\left\vert \lambda_{n}-\lambda_{id}\right\vert \rightarrow0$ as
$n\rightarrow\infty$. Define time-changed processes $\widetilde{\mathsf{S}%
}^{[n]}=(\widetilde{\mathsf{S}}_{t}^{[n]})_{t\in\lbrack0,1]}$ by
$\widetilde{\mathsf{S}}_{t}^{[n]}:=\widehat{\mathsf{S}}_{\lambda_{n}(t)}%
^{[n]}$, and note that
\begin{equation}
d_{\mathcal{J}_{1}}(\widehat{\mathsf{S}}^{[n]},\widetilde{\mathsf{S}}%
^{[n]})\leq\sup\nolimits_{\lbrack0,1]}\left\vert \lambda_{n}-\lambda
_{id}\right\vert \text{\quad on }\Omega\text{ \quad for }n\geq1\text{.}%
\end{equation}
(For any $n\geq1$ and $\omega\in\Omega$, the paths $\mathsf{x}=(\mathsf{x}%
_{t})=(\widehat{\mathsf{S}}_{t}^{[n]}(\omega))$ and $\mathsf{y}=(\mathsf{y}%
_{t})=(\widetilde{\mathsf{S}}_{t}^{[n]}(\omega))$ are related by
$\mathsf{y}=\mathsf{x}\circ\lambda_{n}$. Now recall the definition of
$d_{\mathcal{J}_{1}}$.) This shows that $d_{\mathcal{J}_{1}}(\widehat
{\mathsf{S}}^{[n]},\widetilde{\mathsf{S}}^{[n]})\longrightarrow0$ uniformly on
$\Omega$ as $n\rightarrow\infty$, and therefore convergence of $(\widehat
{\mathsf{S}}^{[n]})$ entails $\widetilde{\mathsf{S}}^{[n]}\Longrightarrow
\widehat{\mathsf{S}}$. In particular,
\begin{equation}
V_{n}+\frac{1}{B_{n}}\left(  A_{s_{n}n}^{(i)}-\frac{\left\lfloor
s_{n}n\right\rfloor }{n}A_{n}^{(i)}\right)  =\widetilde{\mathsf{S}}_{s}%
^{[n]}\Longrightarrow\widehat{\mathsf{S}}_{s}\overset{d}{=}s^{1/\alpha
}\widehat{\mathsf{S}}_{1}+\mathsf{a}_{s}^{(i)}\text{ \quad as }n\rightarrow
\infty\text{.}\label{Eq_SecondObs}%
\end{equation}
But (\ref{Eq_FirstObs}) and (\ref{Eq_SecondObs}), together with continuity of
$t\mapsto\mathsf{a}_{t}^{(i)}$, give (\ref{Eq_TheGoalNow}).
\end{proof}%

\vspace{0.2cm}%
%

\noindent
\textbf{Tightness in }$(\mathcal{D}([0,1],\mathbb{R}^{d}),\mathcal{J}_{1}%
)$\textbf{.} We can now tackle the crucial tightness condition.

\begin{proposition}
[\textbf{Tightness}]\label{P_Tightness}Under the assumptions of Theorem
\ref{T_SLTgivesFSLT}, the sequence $(\mathsf{S}^{[n]})_{n\geq1}$ on
$(X,\mathcal{A},\mu)$ is tight for the (strong) $\mathcal{J}_{1}$-topology on
$\mathcal{D}([0,1],\mathbb{R}^{d})$.
\end{proposition}

\begin{proof}
\textbf{(i)} We check that the distributions under $\mu$ of the $\mathsf{S}%
^{[n]}$ are uniformly tight as a sequence of Borel probability measures on
$\mathcal{D}([0,1],\mathbb{R}^{d})$ equipped with the (strong) $\mathcal{J}%
_{1}$-topology. For $\mathsf{x}\in\mathcal{D}([0,1],\mathbb{R}^{d})$ and
$\delta\in(0,1)$ define
\[
\Delta_{\delta}^{(1)}(\mathsf{x}):=\sup_{0\leq t\leq\delta}\left\Vert
\mathsf{x}_{t}-\mathsf{x}_{0}\right\Vert \text{ \quad and \quad}\Delta
_{\delta}^{(2)}(\mathsf{x}):=\sup_{1-\delta\leq t\leq1}\left\Vert
\mathsf{x}_{1}-\mathsf{x}_{t}\right\Vert \text{,}%
\]
while
\[
\Delta_{\delta}^{(3)}(\mathsf{x}):=\sup_{0\vee(t-\delta)\leq t^{\prime
}<t<t^{\prime\prime}\leq(t+\delta)\wedge1}(\left\Vert \mathsf{x}%
_{t}-\mathsf{x}_{t^{\prime}}\right\Vert \wedge\left\Vert \mathsf{x}%
_{t^{\prime\prime}}-\mathsf{x}_{t}\right\Vert )\text{.}%
\]
A set $\mathcal{K}\subseteq\mathcal{D}([0,1],\mathbb{R}^{d})$ with
$\mathsf{x}_{0}=0$ for all $\mathsf{x}\in\mathcal{K}$ is relatively compact in
the $\mathcal{J}_{1}$-topology iff for every $j\in\{1,2,3\}$,
\begin{equation}
\lim_{\delta\searrow0}\,\sup_{\mathsf{x}\in\mathcal{K}}\,\Delta_{\delta}%
^{(j)}(\mathsf{x})=0\text{,}%
\end{equation}
see statements 2.7.2 and 2.7.3 of \cite{Sk}. As a consequence, uniform
tightness of $(\mathsf{S}^{[n]})_{n\geq1}$ in $(\mathcal{D}([0,1],\mathbb{R}%
^{d}),\mathcal{J}_{1})$ can be verified by showing that for every
$j\in\{1,2,3\}$,
\begin{equation}
\lim_{\delta\searrow0}\,\underset{n\rightarrow\infty}{\overline{\lim}}%
\,\mu\left(  \Delta_{\delta}^{(j)}(\mathsf{S}^{[n]})>\varepsilon\right)
=0\text{ \quad for }\varepsilon>0\text{.}\label{Eq_tightnessCrit}%
\end{equation}
\textbf{(ii)} To efficiently deal with the $A_{n}$ we define new observables
$f_{n}:X\rightarrow\mathbb{R}^{d}$ with $f_{n}:=f-A_{n}/n$, $n\geq1$, so that
$\mathbf{S}_{k}(f_{n})=\mathbf{S}_{k}(f)-(k/n)A_{n}$ whenever $k,n\geq1 $.
Consequently,
\begin{equation}
\mathsf{S}_{t}^{[n]}=\frac{1}{B_{n}}\,\mathbf{S}_{\left\lfloor tn\right\rfloor
}(f_{n})\text{ \quad for }n\geq1\text{ and }t\in\lbrack0,1]\text{,}%
\label{Eq_hdvhdhvhhdshjsdhfdhhvcsdvsggggggg}%
\end{equation}
whereas
\begin{equation}
\mathsf{S}_{t}^{[n]}-\mathsf{S}_{t^{\prime}}^{[n]}=\frac{1}{B_{n}}%
\,\mathbf{S}_{\left\lfloor tn\right\rfloor -\left\lfloor t^{\prime
}n\right\rfloor }(f_{n})\circ T^{\left\lfloor t^{\prime}n\right\rfloor }\text{
\quad for }n\geq1\text{ and }0\leq t^{\prime}\leq t\leq1\text{.}%
\end{equation}
For later use, pick a constant $C>0$ such that $\Pr\left[  \left\Vert
S\right\Vert >t\right]  \leq$ $C\cdot t^{-\alpha}$ for $t\geq1/24$%
.\newline\newline\textbf{(iii)} We first establish (\ref{Eq_tightnessCrit})
for $j=1$. To this end, we fix $\varepsilon>0$ and observe that
(\ref{Eq_hdvhdhvhhdshjsdhfdhhvcsdvsggggggg}) yields $\Delta_{\delta}%
^{(1)}(\mathsf{S}^{[n]})=\max_{1\leq k\leq\delta n}\left\Vert B_{n}%
^{-1}\mathbf{S}_{k}(f_{n})\right\Vert $, which allows us to apply
(\ref{Eq_MI1_a}) of Lemma \ref{L_MaxIneq1} to $g:=f_{n}$. This gives
\begin{align}
\mu\left(  \Delta_{\delta}^{(1)}(\mathsf{S}^{[n]})>\varepsilon\right)   &
\leq\frac{2e^{R}}{a}\,\max_{1\leq k\leq\delta n}\mu\left(  \left\Vert
\frac{\mathbf{S}_{k}(f_{n})}{B_{n}}\right\Vert >\frac{\varepsilon}{4}\right)
\nonumber\\
+  & \delta n\max_{1\leq k\leq\delta n}\mu\left(  \frac{\vartheta_{f_{n},k}%
}{B_{n}}>\frac{\varepsilon}{4}\right)  \text{ \quad for }\delta\in(0,1)\text{
and }n\geq1\text{.}\label{Eq_uturbavwqqq}%
\end{align}
Consider the first expression on the right-hand side of (\ref{Eq_uturbavwqqq}%
). Since $t\mapsto\mathsf{a}_{t}$ is continuous, we can find $\delta^{\prime
}>0$ such that $\left\Vert \mathsf{a}_{k/n}\right\Vert <\varepsilon/16$
whenever $1\leq k\leq\delta^{\prime}n$. By Lemma \ref{L_TheAn}, there is some
$n^{\ast}=n^{\ast}(\varepsilon)$ s.t. $\max_{1\leq k\leq n}\left\Vert
B_{n}^{-1}(A_{k}-\frac{k}{n}A_{n})-\mathsf{a}_{k/n}\right\Vert <\varepsilon
/16$ for $n\geq n^{\ast}$. Now
\begin{equation}
\frac{1}{B_{n}}\mathbf{S}_{k}(f_{n})=\frac{1}{B_{n}}(\mathbf{S}_{k}%
(f)-A_{k})+\mathsf{a}_{\frac{k}{n}}+\frac{1}{B_{n}}\left(  A_{k}-\frac{k}%
{n}A_{n}\right)  -\mathsf{a}_{\frac{k}{n}}\text{,}\label{Eq_qqqqqqqaqqqaw}%
\end{equation}
so that
\begin{equation}
\left\{  \left\Vert \frac{\mathbf{S}_{k}(f_{n})}{B_{n}}\right\Vert
>\frac{\varepsilon}{4}\right\}  \subseteq\left\{  \left\Vert \frac{1}{B_{n}%
}(\mathbf{S}_{k}(f)-A_{k})\right\Vert >\frac{\varepsilon}{8}\right\}  \text{
\quad}%
\begin{array}
[c]{c}%
\text{for }n\geq n^{\ast}\text{ and}\\
1\leq k\leq\delta^{\prime}n\text{.}%
\end{array}
\label{Eq_mnnbbvcxyasdfdgg}%
\end{equation}
Due to (\ref{Eq_BasicAssmThmImpli}), we see that for every $\delta
\in(0,\varepsilon^{\alpha})$ there is some $k^{\ast}(\delta)$ such that
\begin{align}
\mu\left(  \left\Vert \frac{1}{B_{k}}(\mathbf{S}_{k}(f)-A_{k})\right\Vert
>\frac{\varepsilon}{24}\delta^{-\frac{1}{\alpha}}\right)   & \leq\Pr\left[
\left\Vert S\right\Vert >\frac{\varepsilon}{24}\delta^{-\frac{1}{\alpha}%
}\right]  +\delta\nonumber\\
& \leq\left[  C\left(  \frac{\varepsilon}{24}\right)  ^{-\alpha}+1\right]
\,\delta\text{ \quad for }k\geq k^{\ast}(\delta)\text{.}%
\label{Eq_sdygvchsadgvchg}%
\end{align}
Since $(B_{n})\in\mathcal{R}_{1/\alpha}$, there are $k_{\ast}$ and, for each
$\delta\in(0,\varepsilon^{\alpha})$, some $n_{\ast}(\delta)$ for which
\begin{equation}
\frac{B_{n}}{B_{k}}\geq\frac{1}{4}\delta^{-\frac{1}{\alpha}}\text{ \quad for
}n\geq n_{\ast}(\delta)\text{ and }k_{\ast}\leq k\leq\delta n\text{.}%
\label{Eq_jfjsdhsahsadhjsdah}%
\end{equation}
(Indeed, There is some non-decreasing sequence $(B_{n}^{\ast})\in
\mathcal{R}_{1/\alpha}$ for which $B_{n}^{\ast}\sim B_{n}$ as $n\rightarrow
\infty$, see Theorem 1.5.3 of \cite{BGT}. Choose $k_{\ast}$ for which
$B_{k}^{\ast}/B_{k}\in(1/\sqrt{2},\sqrt{2})$ whenever $k\geq k_{\ast}$. Given
$\delta>0$ there is some $n_{\ast}(\delta)$ such that $B_{n}^{\ast}/B_{\delta
n}^{\ast}\geq\delta^{-\frac{1}{\alpha}}/2$ for $n\geq n_{\ast}(\delta)$. Now
(\ref{Eq_jfjsdhsahsadhjsdah}) follows if we write $B_{n}/B_{k}=(B_{n}%
/B_{n}^{\ast})(B_{k}^{\ast}/B_{k})(B_{n}^{\ast}/B_{k}^{\ast})$ and use
$B_{n}^{\ast}/B_{k}^{\ast}\geq B_{n}^{\ast}/B_{\delta n}^{\ast}$ for
$k\leq\delta n$.) Combining (\ref{Eq_sdygvchsadgvchg}) and
(\ref{Eq_jfjsdhsahsadhjsdah}) we find that
\begin{align}
\mu\left(  \left\Vert \frac{1}{B_{n}}(\mathbf{S}_{k}(f)-A_{k})\right\Vert
>\frac{\varepsilon}{8}\right)   & =\mu\left(  \left\Vert \frac{1}{B_{k}%
}(\mathbf{S}_{k}(f)-A_{k})\right\Vert >\frac{\varepsilon}{8}\frac{B_{n}}%
{B_{k}}\right) \nonumber\\
& \leq\left[  C\left(  \frac{\varepsilon}{24}\right)  ^{-\alpha}+1\right]
\,\delta\text{ \quad}%
\begin{array}
[c]{c}%
\text{for }n\geq n^{\ast}\vee n_{\ast}(\delta)\text{ and}\\
k\geq k_{\ast}\vee k^{\ast}(\delta)\text{.}%
\end{array}
\label{Eq_usususushwebw}%
\end{align}
Recalling (\ref{Eq_mnnbbvcxyasdfdgg}) we thus obtain, for all $\delta
\in(0,\delta^{\prime}\wedge\varepsilon^{\alpha})$,
\begin{equation}
\underset{n\rightarrow\infty}{\overline{\lim}}\,\max_{k_{\ast}\vee k^{\ast
}(\delta)\leq k\leq\delta n}\mu\left(  \left\Vert \frac{\mathbf{S}_{k}(f_{n}%
)}{B_{n}}\right\Vert >\frac{\varepsilon}{4}\right)  \leq\left[  C\left(
\frac{\varepsilon}{24}\right)  ^{-\alpha}+1\right]  \,\delta\text{.}%
\label{Eq_bvxbvxnbvxx}%
\end{equation}
On the other hand, (\ref{Eq_mnnbbvcxyasdfdgg}) shows that for every given
$k\geq1$,
\begin{equation}
\left\Vert \frac{\mathbf{S}_{k}(f_{n})}{B_{n}}\right\Vert \overset{\mu
}{\longrightarrow}0\text{ \quad as }n\rightarrow\infty\text{.}%
\end{equation}
Therefore, whatever $\delta\in(0,\delta^{\prime}\wedge\varepsilon^{\alpha})$,
we have
\[
\underset{n\rightarrow\infty}{\overline{\lim}}\,\max_{1\leq k<k_{\ast}\vee
k^{\ast}(\delta)}\mu\left(  \left\Vert \frac{\mathbf{S}_{k}(f_{n})}{B_{n}%
}\right\Vert >\frac{\varepsilon}{4}\right)  =0\text{.}%
\]
Together with (\ref{Eq_bvxbvxnbvxx}) this yields, for every $\delta
\in(0,\delta^{\prime}\wedge\varepsilon^{\alpha})$,
\begin{equation}
\underset{n\rightarrow\infty}{\overline{\lim}}\,\max_{1\leq k\leq\delta n}%
\mu\left(  \left\Vert \frac{\mathbf{S}_{k}(f_{n})}{B_{n}}\right\Vert
>\frac{\varepsilon}{4}\right)  \leq\left[  C\left(  \frac{\varepsilon}%
{24}\right)  ^{-\alpha}+1\right]  \,\delta=:C_{\ast}(\varepsilon
)\,\delta\text{.}\label{Eq_SupiDupi}%
\end{equation}
Turning to the second term on the right-hand side of (\ref{Eq_uturbavwqqq}),
observe first that trivially $\vartheta_{f_{n}}=\vartheta_{f}$ for every
$n\geq1$, and hence $\vartheta_{f_{n},k}=\vartheta_{f,k}$ for all $k,n\geq1$.
But then Lemma \ref{L_ThetaFTail} immediately shows that
\begin{equation}
\underset{n\rightarrow\infty}{\overline{\lim}}\,\delta n\max_{1\leq
k\leq\delta n}\mu\left(  \frac{\vartheta_{f_{n},k}}{B_{n}}>\frac{\varepsilon
}{4}\right)  \leq\zeta_{f}\,\left(  \frac{\varepsilon}{4}\right)  ^{-\alpha
}\delta=:C^{\ast}(\varepsilon)\,\delta\text{.}%
\label{Eq_fhesvbjsvbjhsvbhjbsjhb}%
\end{equation}
When combined with (\ref{Eq_SupiDupi}) this implies (\ref{Eq_tightnessCrit})
for $j=1$.\newline\newline\textbf{(iv)} To deal with (\ref{Eq_tightnessCrit})
for $j=2$, write
\begin{align*}
\Delta_{\delta}^{(2)}(\mathsf{S}^{[n]})  & =\max_{\left\lfloor (1-\delta
)n\right\rfloor \leq j\leq n}\left\Vert \frac{1}{B_{n}}(\mathbf{S}_{n}%
(f_{n})-\mathbf{S}_{j}(f_{n}))\right\Vert \\
& =\left(  \max_{0\leq k\leq n-\left\lfloor (1-\delta)n\right\rfloor
}\left\Vert \frac{1}{B_{n}}(\mathbf{S}_{n-\left\lfloor (1-\delta
)n\right\rfloor }(f_{n})-\mathbf{S}_{k}(f_{n}))\right\Vert \right)  \circ
T^{\left\lfloor (1-\delta)n\right\rfloor }\text{.}%
\end{align*}
Using $T$-invariance of $\mu$, and the inequalities (\ref{Eq_MI1_b}) and
(\ref{Eq_MI1_a}) (again for $g:=f_{n}$), we then find that
\begin{align*}
\mu\left(  \Delta_{\delta}^{(2)}(\mathsf{S}^{[n]})>\varepsilon\right)   &
=\mu\left(  \max_{0\leq k\leq n-\left\lfloor (1-\delta)n\right\rfloor
}\left\Vert \frac{1}{B_{n}}(\mathbf{S}_{n-\left\lfloor (1-\delta
)n\right\rfloor }(f_{n})-\mathbf{S}_{k}(f_{n}))\right\Vert >\varepsilon\right)
\\
& \leq\mu\left(  \max_{0\leq k\leq n-\left\lfloor (1-\delta)n\right\rfloor
}\left\Vert \frac{\mathbf{S}_{k}(f_{n})}{B_{n}}\right\Vert >\frac{\varepsilon
}{2}\right) \\
& \leq\frac{2e^{R}}{a}\,\max_{0\leq k\leq n-\left\lfloor (1-\delta
)n\right\rfloor }\mu\left(  \left\Vert \frac{\mathbf{S}_{k}(f_{n})}{B_{n}%
}\right\Vert >\frac{\varepsilon}{8}\right) \\
& +(n-\left\lfloor (1-\delta)n\right\rfloor )\max_{0\leq k\leq n-\left\lfloor
(1-\delta)n\right\rfloor }\mu\left(  \frac{\vartheta_{f_{n},k}}{B_{n}}%
>\frac{\varepsilon}{8}\right)  \text{.}%
\end{align*}
Since $n-\left\lfloor (1-\delta)n\right\rfloor \sim\delta n$ as $n\rightarrow
\infty$, assertion (\ref{Eq_tightnessCrit}) for $j=2$ follows from
(\ref{Eq_SupiDupi}) and (\ref{Eq_fhesvbjsvbjhsvbhjbsjhb}).\newline%
\newline\textbf{(v)} We finally turn to (\ref{Eq_tightnessCrit}) for $j=3$.
For any $\delta\in(0,1)$, if $n\geq1/\delta$ then $\left\lfloor 2\delta
n\right\rfloor \geq\delta n$, and thus for any triple $t^{\prime}%
<t<t^{\prime\prime}$ as in the definition of $\Delta_{\delta}^{(3)}$, the
points $\left\lfloor t^{\prime}n\right\rfloor \leq\left\lfloor tn\right\rfloor
\leq\left\lfloor t^{\prime\prime}n\right\rfloor $ are contained in an interval
of the form $[k\left\lfloor 2\delta n\right\rfloor ,(k+3)\left\lfloor 2\delta
n\right\rfloor ]$. Consequently, $\Delta_{\delta}^{(3)}(\mathsf{S}^{[n]})$
cannot exceed
\[
\max_{0\leq k\leq n/\left\lfloor 2\delta n\right\rfloor }\left(  \max_{0\leq
i<j<l\leq4\left\lfloor 2\delta n\right\rfloor }\left(  \frac{\left\Vert
\mathbf{S}_{j}(f_{n})-\mathbf{S}_{i}(f_{n})\right\Vert }{B_{n}}\wedge
\frac{\left\Vert \mathbf{S}_{l}(f_{n})-\mathbf{S}_{j}(f_{n})\right\Vert
}{B_{n}}\right)  \right)  \circ T^{k\left\lfloor 2\delta n\right\rfloor
}\text{.}%
\]
Therefore, $\{\Delta_{\delta}^{(3)}(\mathsf{S}^{[n]})>\varepsilon\}$ is
contained in
\[
\bigcup_{0\leq k\leq n/\left\lfloor 2\delta n\right\rfloor }T^{-k\left\lfloor
2\delta n\right\rfloor }\left\{  \max_{0\leq i<j<l\leq4\left\lfloor 2\delta
n\right\rfloor }\left(  \frac{\left\Vert \mathbf{S}_{j}(f_{n})-\mathbf{S}%
_{i}(f_{n})\right\Vert }{B_{n}}\wedge\frac{\left\Vert \mathbf{S}_{l}%
(f_{n})-\mathbf{S}_{j}(f_{n})\right\Vert }{B_{n}}\right)  >\varepsilon
\right\}  \text{,}%
\]
and due to $T$-invariance of $\mu$, we find that for $\delta\in(0,1/2)$ and
$n$ so large that $1+n/\left\lfloor 2\delta n\right\rfloor \leq1/\delta$,
\[
\mu\left(  \Delta_{\delta}^{(3)}(\mathsf{S}^{[n]})>\varepsilon\right)
\leq\frac{1}{\delta}\,\mu\left(  \max_{0\leq i<j<l\leq4\left\lfloor 2\delta
n\right\rfloor }\left(  \frac{\left\Vert \mathbf{S}_{j}(f_{n})-\mathbf{S}%
_{i}(f_{n})\right\Vert }{B_{n}}\wedge\frac{\left\Vert \mathbf{S}_{l}%
(f_{n})-\mathbf{S}_{j}(f_{n})\right\Vert }{B_{n}}\right)  >\varepsilon\right)
\text{.}%
\]
Hence, applying (\ref{Eq_MI1_ccc}) to $g:=f_{n}$,
\begin{align*}
\mu\left(  \Delta_{\delta}^{(3)}(\mathsf{S}^{[n]})>\varepsilon\right)   &
\leq\frac{e^{R}}{\delta a}\mu\left(  \max_{1\leq k\leq8\delta n}\left\Vert
\frac{\mathbf{S}_{k}(f_{n})}{B_{n}}\right\Vert >\frac{\varepsilon}{4}\right)
\,\\
& \cdot\left[  \mu\left(  \max_{1\leq k\leq8\delta n}\left\Vert \frac
{\mathbf{S}_{k}(f_{n})}{B_{n}}\right\Vert >\frac{\varepsilon}{4}\right)
+8\delta n\max_{1\leq k\leq8\delta n}\mu\left(  \frac{\vartheta_{f_{n},k}%
}{B_{n}}>\frac{\varepsilon}{4}\right)  \right]  \text{.}%
\end{align*}
In view of (\ref{Eq_SupiDupi}) and (\ref{Eq_fhesvbjsvbjhsvbhjbsjhb}) this
shows that for every $\delta\in(0,\varepsilon^{\alpha}\wedge\frac{1}{2}) $,%
\[
\underset{n\rightarrow\infty}{\overline{\lim}}\,\mu\left(  \Delta_{\delta
}^{(3)}(\mathsf{S}^{[n]})>\varepsilon\right)  \leq\frac{e^{R}}{a}C_{\ast
}(\varepsilon)[C_{\ast}(\varepsilon)+8C^{\ast}(\varepsilon)]\,\delta\text{,}%
\]
and (\ref{Eq_tightnessCrit}) with $j=3$ follows as $\delta\searrow0$.
\end{proof}%

\vspace{0.2cm}%

We finally provide a $d$-dimensional version of Corollary 3 of \cite{Z7}.

\begin{lemma}
[{\textbf{Strong distributional convergence of }$\mathsf{S}^{[n]}
$\textbf{\ in} $(\mathcal{D}([0,\infty),\mathbb{R}^{d}),\mathcal{J}_{1})$}%
]\label{L_StrongDistCge}Let $T$ be an ergodic measure preserving map on the
probability space $(X,\mathcal{A},\mu)$, $f:X\rightarrow\mathbb{R}^{d}$ a
measurable function, and $\mathsf{S}^{[n]}$ defined as in
(\ref{Eq_DefPSprocesses2}) with $(A_{n},B_{n})\in\mathbb{R}^{d}\times
(0,\infty) $ satisfying $B_{n}\rightarrow\infty$ and $\left\Vert
A_{n}\right\Vert =o(nB_{n})$ as $n\rightarrow\infty$. Then
\begin{equation}
d_{\mathcal{J}_{1},\infty}(\mathsf{S}^{[n]},\mathsf{S}^{[n]}\circ
T)\overset{\mu}{\longrightarrow}0\text{ \quad as }n\rightarrow\infty
\text{.}\label{Eq_bvbvbvbvbvbvbfshbfuserb}%
\end{equation}
Consequently, whenever $\mathsf{S}$ is a random element of $(\mathcal{D}%
([0,\infty),\mathbb{R}^{d}),\mathcal{J}_{1})$, then
\begin{equation}
\mathsf{S}^{[n]}\overset{\nu}{\Longrightarrow}\mathsf{S}\text{ \quad in
}(\mathcal{D}([0,\infty),\mathbb{R}^{d}),\mathcal{J}_{1})
\end{equation}
holds for \emph{all} probabilities $\nu\ll\mu$ as soon as it holds for
\emph{some} $\nu$.
\end{lemma}

\begin{proof}
\textbf{(i)} Below we shall show that for every $\varepsilon>0$ there are
constants $\varrho_{\varepsilon}(n)$ such that $\varrho_{\varepsilon
}(n)\rightarrow0$ as $n\rightarrow\infty$ while
\begin{equation}
\mu\left(  d_{\mathcal{J}_{1},s}(\mathsf{S}^{[n]},\mathsf{S}^{[n]}\circ
T)>\varepsilon\right)  \leq\varrho_{\varepsilon}(n)\text{ \quad for }%
n\geq1\text{ and }s\in\lbrack0,\infty)\text{.}\label{Eq_vsvsvsvsvsqqqqqqqqqq}%
\end{equation}
Therefore, whatever $\varepsilon>0$,
\begin{multline*}
\int_{0}^{\infty}\int_{X}e^{-s}(1\wedge d_{\mathcal{J}_{1},s}(\mathsf{S}%
^{[n]},\mathsf{S}^{[n]}\circ T))\,d\mu\,ds\\
\leq\int_{0}^{\infty}e^{-s}\left[  \mu\left(  d_{\mathcal{J}_{1},s}%
(\mathsf{S}^{[n]},\mathsf{S}^{[n]}\circ T)>\varepsilon\right)  +\varepsilon
\right]  \,ds\\
\leq\varrho_{\varepsilon}(n)+\varepsilon<2\varepsilon\text{ \quad for }n\geq
n_{1}(\varepsilon)\text{,}%
\end{multline*}
proving that $\int_{X}d_{\mathcal{J}_{1},\infty}(\mathsf{S}^{[n]}%
,\mathsf{S}^{[n]}\circ T)\,d\mu\rightarrow0$, which implies
(\ref{Eq_bvbvbvbvbvbvbfshbfuserb}). Now Theorem 1 of \cite{Z7} immediately
yields the second assertion of our lemma, about strong distributional
convergence of $(\mathsf{S}^{[n]})$.\newline\newline\textbf{(ii)}\ For
$n\geq4$ define a time-change $\lambda_{n}\in\Lambda$ by affinely
interpolating between $\lambda_{n}(0)=0$, $\lambda_{n}(1/n)=2/n $,
$\lambda_{n}((n-2)/n)=(n-1)/n$, and $\lambda_{n}(1)=1$. By straightforward
elementary considerations,
\begin{align*}
d_{\mathcal{J}_{1},1}(\mathsf{S}^{[n]},\mathsf{S}^{[n]}\circ T)  & \leq
\sup_{t\in\lbrack0,1]}\left\Vert \mathsf{S}_{\lambda_{n}(t)}^{[n]}%
-(\mathsf{S}^{[n]}\circ T)_{t}\right\Vert \\
& \leq\frac{2}{n}+\frac{2\left\Vert A_{n}\right\Vert }{nB_{n}}+\frac
{\left\Vert f\right\Vert +\left\Vert f\circ T^{n-1}\right\Vert +\left\Vert
f\circ T^{n}\right\Vert }{B_{n}}\text{.}%
\end{align*}
Taking $n_{0}(\varepsilon)$ so large that $n\geq n_{0}(\varepsilon)$ implies
$2(B_{n}+\left\Vert A_{n}\right\Vert )/(nB_{n})<\varepsilon/2$,
\begin{multline*}
\mu\left(  d_{\mathcal{J}_{1},1}(\mathsf{S}^{[n]},\mathsf{S}^{[n]}\circ
T)>\varepsilon\right)  \leq\mu\left(  \frac{\left\Vert f\right\Vert
+\left\Vert f\circ T^{n-1}\right\Vert +\left\Vert f\circ T^{n}\right\Vert
}{B_{n}}>\varepsilon B_{n}/2\right) \\
\leq\mu\left(  \frac{\left\Vert f\right\Vert }{B_{n}}>\frac{\varepsilon}%
{6}\right)  +\mu\left(  \frac{\left\Vert f\circ T^{n-1}\right\Vert }{B_{n}%
}>\frac{\varepsilon}{6}\right)  +\mu\left(  \frac{\left\Vert f\circ
T^{n}\right\Vert }{B_{n}}>\frac{\varepsilon}{6}\right) \\
\leq3\,\mu\left(  \frac{\left\Vert f\right\Vert }{B_{n}}>\frac{\varepsilon}%
{6}\right)  \text{ \quad for }n\geq n_{0}(\varepsilon)\text{.}%
\end{multline*}
But $\mu(\frac{\left\Vert f\right\Vert }{B_{n}}>\frac{\varepsilon}%
{6})\rightarrow0$ since $\left\Vert f\right\Vert /B_{n}\rightarrow0$ and $\mu$
is finite.

The same argument, \emph{with the same upper bound}, works if the time
interval $[0,1]$ is replaced by any $[0,s]$ with $s>0$. This proves
(\ref{Eq_vsvsvsvsvsqqqqqqqqqq}).
\end{proof}%

\vspace{0.2cm}%

We can now wrap up the

\begin{proof}
[\textbf{Proof of Theorem \ref{T_SLTgivesFSLT}}]According to Lemma
\ref{L_StrongDistCge} it suffices to prove $\mathsf{S}^{[n]}\overset{\nu
}{\Longrightarrow}\mathsf{S}$ in $(\mathcal{D}([0,\infty),\mathbb{R}%
^{d}),\mathcal{J}_{1})$ with respect to $\nu=\mu$, distributional convergence
$\mathsf{S}^{[n]}\overset{\nu}{\Longrightarrow}\mathsf{S}$ for arbitrary
$\nu\ll\mu$ then being automatic.

Due to stochastic continuity of $\mathsf{S}$ we see that $\mathsf{S}%
^{[n]}\overset{\mu}{\Longrightarrow}\mathsf{S}$ in $(\mathcal{D}%
([0,\infty),\mathbb{R}^{d}),\mathcal{J}_{1})$ follows as soon as
$\mathsf{S}^{[n]}\overset{\mu}{\Longrightarrow}\mathsf{S}$ in $(\mathcal{D}%
([0,s],\mathbb{R}^{d}),\mathcal{J}_{1})$ for every $s>0$. Now Propositions
\ref{P_Finidim} and \ref{P_Tightness} immediately show, via Prohorov's
Theorem, that $\mathsf{S}^{[n]}\overset{\mu}{\Longrightarrow}\mathsf{S}$ in
$(\mathcal{D}([0,1],\mathbb{R}^{d}),\mathcal{J}_{1})$, and the argument for
$s\neq1$ is the same.
\end{proof}%

\vspace{0.2cm}%

\section{Proof of the concrete limit theorems}

The one-dimensional case does not present any difficulties.

\begin{proof}
[\textbf{Proof of Theorem \ref{T_FSL_GM}}]The assumption that $f$ be in the
domain of attraction of $S$ and the condition
(\ref{Eq_ConditionRegularityTail}) on the tail of $\vartheta_{f}$ together
imply that
\begin{equation}
\int\vartheta_{f}^{\eta}\,d\mu<\infty\text{ \quad for }\eta\in(0,1\wedge
\alpha)\text{.}%
\end{equation}
Finiteness of $\int\vartheta_{f}^{\eta}\,d\mu$ for some $\eta\in(0,1]$,
however, is the fundamental regularity assumption of \cite{Goue2}. Theorem 1.5
of \cite{Goue2} asserts that if such an observable $f$ is in the domain of
attraction of $S$, then its ergodic sums satisfy a a stable limit theorem as
in (\ref{Eq_ADADAD}), with constants $(A_{n})$ and $(B_{n})$ obtained from the
law of $f$ in exactly the same way as in the iid case. In particular, the
assumptions of Theorem \ref{T_SLTgivesFSLT} is fulfilled, and
(\ref{Eq_FSLTinThm}) follows.
\end{proof}%

\vspace{0.2cm}%

Turning to the specific vector-valued scenario of Theorem \ref{T_FSL_GMddim},
we recall (\ref{Eq_RegVarOfStableTails}). It is then easy to see (Example
2.3.5 of \cite{ST}) that if a full stable vector is of the form $S=(S^{(1)}%
,\ldots,S^{(d)})$ with independent scalar $\alpha$-stable variables $S^{(j)}$
which satisfy $\mu(S^{(j)}>t)\sim c_{+}^{(j)}t^{-\alpha}\ell(t)$%
\ and$\ \mu(S^{(j)}<-t)\sim c_{-}^{(j)}t^{-\alpha}\ell(t)$ as $t\rightarrow
\infty$, then the spectral measure $\Lambda$ of $S$ coincides with
\begin{equation}
\Lambda_{(c_{+}^{(j)},c_{-}^{(j)})_{j=1}^{d}}=\sum_{j=1}^{d}\left(
c_{+}^{(j)}\delta_{e_{j}}+c_{-}^{(j)}\delta_{-e_{j}}\right)  \text{,}%
\label{Eq_IndepSpectralMeasure}%
\end{equation}
where $e_{1}=(1,0,\ldots,0),\ldots,e_{d}=(0,\ldots,0,1)$ are the standard
basis of $\mathbb{R}^{d}$.%

\vspace{0.2cm}%

\begin{proof}
[\textbf{Proof of Theorem \ref{T_FSL_GMddim}}]\textbf{(i)} Because of Theorem
\ref{T_SLTgivesFSLT}, the main point is to prove%
\begin{equation}
R_{n}:=\frac{1}{B_{n}}\,(\mathbf{S}_{n}(f)-A_{n})\overset{\mu}{\Longrightarrow
}S\text{ \quad as }n\rightarrow\infty\label{Eq_vhbjsdvbdfjbvjhsdjjjjjjjjj}%
\end{equation}
with $S=(S^{(1)},\ldots,S^{(d)})$ an independent tuple as in the statement of
the present theorem. In view of (\ref{Eq_DefBn}) we have $(B_{n}^{(j)}%
)=(B_{n})$ for each $j\in\{1,\ldots,d\}$, and Theorem \ref{T_FSL_GM}
immediately gives
\begin{equation}
\frac{1}{B_{n}}\,(\mathbf{S}_{n}(f^{(j)})-A_{n}^{(j)})\overset{\mu
}{\Longrightarrow}S^{(j)}\text{ \quad as }n\rightarrow\infty\text{.}%
\label{Eq_IndividComponent}%
\end{equation}
As a consequence, the sequence $(R_{n})_{n\geq1}$ of random vectors is tight,
and every distributional limit point is a full random vector. Therefore
(\ref{Eq_vhbjsdvbdfjbvjhsdjjjjjjjjj}) follows once we show that for every
sequence of indices $n_{k}\nearrow\infty$ such that $R_{n_{k}}\overset{\mu
}{\Longrightarrow}R$ for some random vector $R$, that limit necessarily
satisfies $R\overset{d}{=}S$.

Below we prove that $\mu\circ f^{-1}\in\mathrm{DOA}(S)$. By Proposition
\ref{Prop_GMvsIID} this implies that
\[
\frac{1}{B_{n}^{\ast}}\,(\mathbf{S}_{n}(f)-A_{n}^{\ast})\overset{\mu
}{\Longrightarrow}S\text{ \quad as }n\rightarrow\infty
\]
for a suitable sequence $(A_{n}^{\ast},B_{n}^{\ast})_{n\geq1}$ in
$\mathbb{R}^{d}\times(0,\infty)$. By the $d$-dimensional convergence of types
principle (Theorem 2.3.17 of \cite{MS}) and (\ref{Eq_IndividComponent}) this
ensures that $R\overset{d}{=}S$ as required.\newline\newline\textbf{(ii)} In
view of (\ref{Eq_IndepSpectralMeasure}) proving our earlier claim $\mu\circ
f^{-1}\in\mathrm{DOA}(S)$ only requires us to check that
\begin{equation}
\mu_{\{\left\Vert f\right\Vert >t\}}\left(  \frac{f}{\left\Vert f\right\Vert
}\in D\,\right)  \longrightarrow\frac{\Lambda_{(c_{+}^{(j)},c_{-}^{(j)}%
)_{j=1}^{d}}(D)}{\sum_{i=1}^{d}(c_{+}^{(i)}+c_{-}^{(i)})}\text{ \quad as
}t\rightarrow\infty\text{,}%
\end{equation}
for Borel $D\subseteq\mathbb{S}^{d-1}$ whose boundary \emph{in} $\mathbb{S}%
^{d-1}$ is disjoint from $\{\pm e_{j}\}_{j=1}^{d}$. This follows as soon as we
prove that for every $j\in\{1,\ldots,d\}$ and $\delta>0$,
\begin{equation}
\mu_{\{\left\Vert f\right\Vert >t\}}\left(  \mathrm{sgn}f^{(j)}=\pm1\text{ and
}\left\vert f^{(i)}\right\vert \leq\delta\left\vert f^{(j)}\right\vert \text{
for }i\neq j\,\right)  \longrightarrow\frac{c_{\pm}^{(j)}}{\sum_{i=1}%
^{d}(c_{+}^{(i)}+c_{-}^{(i)})}\text{.}\label{Eq_chdsjdvcjsvdaaaa}%
\end{equation}
To this end, fix $j$ and observe that for $t>\sqrt{d}M$ our assumption
(\ref{Eq_DisjointTailSupports}) ensures that
\[
\{f^{(j)}>t\}\subseteq\{\left\Vert f\right\Vert >t\text{ and }f^{(j)}%
>M\}\subseteq\{f^{(j)}>\sqrt{t^{2}-dM^{2}}\}\text{.}%
\]
Since $\sqrt{t^{2}-dM^{2}}\sim t$ the tail assumption
(\ref{Eq_AssmIndividTails}) yields
\begin{equation}
\mu(\left\Vert f\right\Vert >t\text{ and }f^{(j)}>M)\sim c_{+}^{(i)}%
t^{-\alpha}\ell(t)\text{ \quad as }t\rightarrow\infty\text{,}\label{Eq_xxxxb}%
\end{equation}
and in the same manner we obtain
\begin{equation}
\mu(\left\Vert f\right\Vert >t\text{ and }f^{(j)}<-M)\sim c_{-}^{(i)}%
t^{-\alpha}\ell(t)\text{ \quad as }t\rightarrow\infty\text{.}\label{Eq_xxxxbb}%
\end{equation}
On the other hand, (\ref{Eq_DisjointTailSupports}) also guarantees that for
$t>\sqrt{d}M$,%
\begin{equation}
\{\left\Vert f\right\Vert >t\}=\bigcup_{\sigma=\pm1}\bigcup_{j=1}%
^{d}\{\left\Vert f\right\Vert >t\text{ and }\left\vert f^{(j)}\right\vert
>M\text{ with }\mathrm{sgn}f^{(j)}=\sigma\}\text{,}\label{Eq_sdgfvgfggggggg}%
\end{equation}
which is a disjoint union. Hence,
\begin{equation}
\mu(\left\Vert f\right\Vert >t)\sim\left(  \sum_{i=1}^{d}(c_{+}^{(i)}%
+c_{-}^{(i)})\right)  \,t^{-\alpha}\ell(t)\text{ \quad as }t\rightarrow
\infty\text{.}\label{Eq_xxxxv}%
\end{equation}
But (\ref{Eq_sdgfvgfggggggg}) also implies that for any $\delta>0$ we have
\begin{align*}
\{\left\Vert f\right\Vert  & >t\text{ and }\left\vert f^{(i)}\right\vert
\leq\delta\left\vert f^{(j)}\right\vert \text{ for }i\neq j\,\text{\ while
}\mathrm{sgn}f^{(j)}=\pm1\}=\\
& =\{\left\Vert f\right\Vert >t\text{ and }\left\vert f^{(j)}\right\vert
>M\text{ with }\mathrm{sgn}f^{(j)}=\pm1\}\text{ \quad for }t\geq t_{0}%
(\delta)\text{.}%
\end{align*}
In view of (\ref{Eq_xxxxb}), (\ref{Eq_xxxxbb}) and (\ref{Eq_xxxxv}) this
validates (\ref{Eq_chdsjdvcjsvdaaaa}).
\end{proof}%

\vspace{0.2cm}%

\section{Two applications}

We illustrate the use of our general results in two specific
situations.\newline\newline\textbf{An arcsine law for some }$\mathbb{Z}%
$\textbf{-extensions of Gibbs-Markov systems.} We first mention an application
of Theorem \ref{T_FSL_GM} to the infinite measure preserving dynamical system
given by the dynamically defined random walk $(\mathbf{S}_{n}(f))_{n\geq0}$.
More precisely, let $(X,\mathcal{A},\mu,T,\xi)$ and $f$ be as in Theorem
\ref{T_FSL_GM}, where we assume for simplicity that $f$ is integer-valued.
Define the \emph{skew product} transformation
\begin{equation}
T_{f}:X\times\mathbb{Z\rightarrow}X\times\mathbb{Z}\text{, \quad}%
T_{f}(x,m):=(T(x),m+f(x))\text{.}%
\end{equation}
The system $(X\times\mathbb{Z},\mathcal{A}\otimes\mathcal{P}(\mathbb{Z}%
),\mu\otimes\iota,T_{f})$ is the $\mathbb{Z}$\emph{-extension of}
$(X,\mathcal{A},\mu,T)$ \emph{by} $f$ (see Chapter 8 of \cite{A}). The
infinite but $\sigma$-finite measure $\mu\otimes\iota$ (with $\iota$ denoting
counting measure on $\mathbb{Z}$) is invariant under $T_{f}$.

Various interesting properties of classical random walks hold for more general
infinite measure preserving systems. In the case of the simplest symmetric
random walk on $\mathbb{Z}$ several relevant quantities converge to the
arcsine law with distribution function $t\mapsto(2/\pi)\arcsin\sqrt{t}$,
$t\in\lbrack0,1]$. These well-known results (Chapter 3 of \cite{F}) can be
generalized in different ways. Classical probability theory provides, for
example, arcsine-type limit theorems for the time of the last visit to a
reference set of finite measure (\cite{Dy} and \cite{L2}), for occupation
times of infinite-measure sets separated from their complement by some
finite-measure set (\cite{L1}), and for occupation times of a half-line under
a random walk (\cite{Sp}). The first two types have been extended to more
general infinite measure preserving systems, see \cite{Th1}, \cite{Th2},
\cite{TZ}, \cite{Z6}, \cite{KZ} and \cite{SY}, while it sems that the last
(Spitzer's arcsine law, which requires a random walk or skew-product
structure) has not. We therefore take the opportunity to point out that the
FSLT above entails a result of this flavour: For $\rho\in(0,1)$ let
$\mathsf{A}_{\rho}$ denote a $[0,1]$-valued random variable which has the
\emph{generalized arcsine law with parameter} $\rho$, that is,
\begin{equation}
\Pr[0\leq\mathsf{A}_{\rho}\leq t]=\frac{\sin\rho\pi}{\pi}\int_{0}^{t}%
s^{\rho-1}(1-s)^{-\rho}\,ds\text{ \qquad for \thinspace}t\in\lbrack
0,1]\text{.}%
\end{equation}
We then obtain

\begin{theorem}
[\textbf{Arcsine Law for }$\mathbb{Z}$\textbf{-extensions of Gibbs-Markov
systems}]\label{T_SpitzerTypeAsinus}Let $(X,\mathcal{A},\mu,T,\xi)$ be a
mixing probability preserving Gibbs-Markov system, and $f:X\rightarrow
\mathbb{Z}$ an observable in the domain of attraction of some $\alpha$-stable
random variable $S$, $\alpha\in(0,2)$. Assume also that $\mu(\vartheta
_{f}>t)=O(\mu(\left\vert f\right\vert >t))$ and that $f$ is centered, $\int
f\,d\mu=0$, in case $\alpha>1$, and symmetrically distributed, $\mu
(f>t)=\mu(f<-t)$, in case $\alpha=1$. Then
\begin{equation}
\frac{1}{n}\sum_{k=0}^{n-1}1_{X\times\mathbb{N}}\circ T_{f}^{k}\overset{\eta
}{\Longrightarrow}\mathsf{A}_{\rho}\text{ \quad as }n\rightarrow
\infty\label{Eq_Asinus}%
\end{equation}
for all probability measures $\eta\ll\mu\otimes\iota$, where $\rho:=\Pr[S>0]$.
\end{theorem}

\begin{proof}
\textbf{(i)} An arbitrary probability measure $\eta\ll\mu\otimes\iota$ can be
represented as $\eta=\sum_{m\in\mathbb{Z}}p_{m}\,\nu_{m}\otimes\delta_{m}$
(with $\delta_{m}$ denoting unit point mass at $m$) for probabilities $\nu
_{m}\ll\mu$ and weights $p_{m}\geq0$ with $\sum_{m\in\mathbb{Z}}p_{m}=1$. It
is straightforward that (\ref{Eq_Asinus}) follows once we prove for that all
for $m\in\mathbb{Z}$ and $\nu_{m}\ll\mu$,
\begin{equation}
\frac{1}{n}\sum_{k=0}^{n-1}1_{X\times\mathbb{N}}\circ T_{f}^{k}\overset
{\nu_{m}\otimes\delta_{m}}{\Longrightarrow}\mathsf{A}_{\rho}\text{ \quad as
}n\rightarrow\infty\text{.}\label{Eq_AsinusOnLevelm}%
\end{equation}
Fix $m\in\mathbb{Z}$ and any probability $\nu_{m}\ll\mu$. Clearly, $T_{f}%
^{n}(x,m)=(T^{n}(x),m+\mathbf{S}_{n}(f)(x))$ for $n\geq0$. As $f$ is centered
in case $\alpha>1$, and symmetrically distributed in case $\alpha=1$, the
canonical normalizing sequence $(A_{n},B_{n})$ for $[\alpha,c_{+},c_{-}]$ is
of the form $(0,B_{n})$. Therefore we see that $T_{f}^{n}(x,m)\in
X\times\mathbb{N}$ iff $\mathsf{S}_{k/n}^{[n]}(x)+m/B_{n}>0$, where
$\mathsf{S}^{[n]}$ is as in (\ref{Eq_DefPSprocesses2}). Defining
$\psi:\mathcal{D}[0,\infty)\rightarrow\mathbb{R}$ by $\psi(\mathsf{x}%
):=\int_{0}^{1}1_{(0,\infty)}(\mathsf{x}_{s})\,ds$, we get
\[
\frac{1}{n}\sum_{k=0}^{n-1}1_{X\times\mathbb{N}}\circ T_{f}^{k}(x,m)=\psi
(\mathsf{S}^{[n]}(x)+m/B_{n})\text{,}%
\]
so that (\ref{Eq_AsinusOnLevelm}) is equivalent to
\begin{equation}
\psi(\mathsf{S}^{[n]}+m/B_{n})\overset{\nu_{m}}{\Longrightarrow}%
\mathsf{A}_{\rho}\text{ \quad as }n\rightarrow\infty\text{.}%
\label{Eq_AsinusForSumOnLevel}%
\end{equation}
\textbf{(ii)} Now Theorem \ref{T_FSL_GM} applies to $f$, and since
$B_{n}\rightarrow\infty$, we have
\[
\mathsf{S}^{[n]}+m/B_{n}\overset{\nu_{m}}{\Longrightarrow}\mathsf{S}\text{
\quad in }(\mathcal{D}[0,\infty),\mathcal{J}_{1})\text{ \quad as }%
n\rightarrow\infty\text{,}%
\]
where $\mathsf{S}=(\mathsf{S}_{t})_{t\geq0}$ is the $\alpha$-stable motion
with $\mathsf{S}_{1}\overset{d}{=}S$. Observe then that while the map $\psi$
is not continuous (for example, it is discontinuous at $\mathsf{x}:=0$), we
have that
\begin{equation}
\psi\text{ is }\mathcal{J}_{1}\text{-continuous at almost every path of
}\mathsf{S}\text{.}\label{Eq_CtyOfpsi}%
\end{equation}
This follows by the argument presented in Appendix M15 of \cite{Bi}, because
$\Pr[\mathsf{S}_{s}=0]=\Pr[S=0]=0$ for all $s>0$. In view of
(\ref{Eq_CtyOfpsi}) the continuous mapping theorem (Theorem 2.7 in \cite{Bi})
immediately shows that
\begin{equation}
\mathsf{S}^{[n]}+m/B_{n}\overset{\nu_{m}}{\Longrightarrow}\mathsf{S}\text{
\quad implies \quad}\psi(\mathsf{S}^{[n]}+m/B_{n})\overset{\nu_{m}%
}{\Longrightarrow}\psi(\mathsf{S})\text{.}%
\end{equation}
By Theorem VI.13 of \cite{Be}, however, $\psi(\mathsf{S})\overset{d}%
{=}\mathsf{A}_{\rho}$, which completes the proof.
\end{proof}%

\vspace{0.2cm}%

\begin{remark}
Note that (\ref{Eq_Asinus}) holds for all probabilities $\eta\ll\mu
\otimes\iota$, even if $T_{f}$ is not ergodic w.r.t. $\mu\otimes\iota$. (Which
happens, for example, if $f$ takes its values in the set $2\mathbb{Z}$ of even
integers.) This is in contrast to the other types of arcsine laws for infinite
measure preserving systems mentioned before, where convergence to the same
limit law under all probabilities absolutely continuous w.r.t. the invariant
measure depends on ergodicity of the system (via the device discussed in
\cite{Z7}).
\end{remark}

\begin{remark}
The regularity condition $\mu(\vartheta_{f}>t)=O(\mu(\left\vert f\right\vert
>t))$ \ allows for integer-valued observables $f$ which need not be constant
on cylinders of any fixed rank.
\end{remark}%

\vspace{0.2cm}%

\noindent
\textbf{Excursions to cusps of intermittent maps.} Interval maps with
indifferent fixed points constitute a basic class of dynamical systems at the
edge of hyperbolicity which has been studied extensively in the last decades.
Their dynamics can be viewed as being driven by a uniformly hyperbolic induced
map, with delays caused by long excursions to the vicinity of the neutral
points. The results of the present paper can be used to clarify questions
about the asymptotic (in)dependence between excursion processes to individual
indifferent fixed points.

To limit technicalities, we focus on the prototypical situation of maps $T$ on
$X:=[0,1]$ with two full branches and neutral fixed points at $x=0$ and $x=1$.
Specifically, assume that there is some $c\in(0,1)$, defining cylinders
$Z_{1}:=(0,c)$ and $Z_{2}:=(c,1)$, such that (with $\lambda$ denoting Lebesgue measure)

\begin{enumerate}
\item[a)] $T\mid_{Z_{0}}$ is an increasing homeomorphism of $(0,c)$ onto
$(0,1)$,

\item[b)] $T\mid_{Z_{0}}$ extends to a $\mathcal{C}^{2}$-map of $(0,c]$ onto
$(0,1]$, with $T^{\prime}>1$ on $(0,c]$, while $T^{\prime}x\rightarrow1$ as
$x\searrow0$, and $T^{\prime}$ is increasing on some neighbourhood of $0$,

\item[c)] there is some decreasing function $\gamma:(0,c]\rightarrow
\lbrack0,\infty)$ for which
\[
\int_{0}^{c}\gamma\,d\lambda<\infty\text{ \quad and \quad}\left\vert
T^{\prime\prime}\right\vert \leq\gamma\text{ on }(0,c]\text{,}%
\]

\item[d)] the map $\widetilde{T}x:=1-T(1-x)$ on $\widetilde{Z}_{0}%
:=(0,\widetilde{c})$ with $\widetilde{c}:=1-c$ satisfies all the conditions
which a)-c) impose on $T\mid_{Z_{0}}$.
\end{enumerate}%

\noindent
Define $Y=Y(T):=(y_{0},y_{1})\subseteq X$ where $y_{0}$ is the unique point of
period $2$ in $Z_{0}$ and $y_{1}:=Ty_{0}$. Set $\varphi_{Y}(x):=\inf
\{n\geq1:T^{n}x\in Y\}$, the \emph{first return time} of $Y$, and let
$T_{Y}x:=T^{\varphi(x)}x$ define the \emph{first return map} $T_{Y}%
:Y\rightarrow Y$ of $Y$, which comes with a natural partition $\xi
_{Y}:=\{Y\cap T^{-1}Z_{j}\cap\{\varphi_{Y}=k\}:j\in\{0,1\},k\geq1\}$. The
function $\varphi_{Y}$ is obviously constant on elements of $\xi_{Y}$. Its
distribution depends on the details of $T$ near the indifferent fixed points,
which can be expressed in terms of $r_{0}(x):=Tx-x$, $x\in(0,c)$ and
$r_{1}(x):=\widetilde{T}x-x$, $x\in(0,\widetilde{c})$.

The following collects some basic facts about such maps.

\begin{proposition}
[\textbf{Basic ergodic properties of intermittent maps}]%
\label{P_BasiXIntermittMap}\textbf{a)} Any map $T$ satisfying a)-d) is
conservative ergodic and exact with respect to Lebesgue measure $\lambda$ and
preserves a $\sigma$-finite Borel measure $\mu\ll\lambda$ with a strictly
positive density $h$ continuous on $(0,1)$. Moreover, the first return map
$(Y,\mathcal{B}_{Y},\mu_{Y},T_{Y},\xi_{Y})$ is a probability preserving
Gibbs-Markov system.{}

\textbf{b)} Assume, in addition, that $\ell$ is regularly varying at $0^{+}$
and that $\kappa_{0},\kappa_{1},p\in(0,\infty)$ are constants such that
\begin{equation}
r_{j}(x)\sim\kappa_{j}\,x^{1+p}\ell(x)\text{ \quad as }x\searrow
0\text{,}\label{Eq_LocalBehAtFP}%
\end{equation}
for $j\in\{0,1\}$. Then
\begin{equation}
\mu(X)<\infty\text{ \quad iff \quad}\int_{0}^{c}\frac{x\,dx}{r_{0}(x)}%
<\infty\text{,}%
\end{equation}
and the functions $\varphi_{Y}^{(0)}:=1_{Y\cap T^{-1}Z_{0}}\varphi_{Y}$ and
$\varphi_{Y}^{(1)}:=1_{Y\cap T^{-1}Z_{1}}\varphi_{Y}$ on $Y$ which record the
durations of excursions from $Y$ to $Z_{0}$ and $Z_{1}$, respectively,
satisfy
\begin{equation}
\mu_{Y}\left(  \varphi_{Y}^{(j)}>m\right)  \sim Cc_{+}^{(j)}A^{-1}(m)\text{
\quad as }m\rightarrow\infty\text{,}\label{Eq_TailOfLocalReturnTimes}%
\end{equation}
where $A^{-1}$ is regularly varying of index $-1/p$ and asymptotically inverse
to $A(t):=t/r_{0}(t)$, the constant $C:=h(c)/[p^{1/p}\int_{Y}h\,d\mu
]\in(0,\infty)$ does not depend on $j$, and $c_{+}^{(0)}:=\kappa_{0}%
^{p}/T^{\prime}(c^{+})$ while $c_{+}^{(1)}:=\kappa_{1}^{p}/T^{\prime}(c^{-})$.
\end{proposition}

\begin{proof}
These are well known facts, see \cite{A}, \cite{Goue}, \cite{Th0a}, \cite{Th3}
and \cite{ZStochDyn}.
\end{proof}%

\vspace{0.2cm}%

It is the joint behaviour of consecutive excursions from $Y$ which determines
interesting aspects of the long-term behaviour of $T$. For example, in the
infinite measure case, the Arcsine Law for occupation times of neighbourhoods
of the neutral fixed points (first established in \cite{Th2}, see also
\cite{TZ} and \cite{Z6}) is best viewed as a consequence of the asymptotic
independence between the excursion processes to $x=0$ and $x=1 $,
respectively. This is made explicit in \cite{Se}, which contains the
$\alpha\in(0,1)$ case of Theorem \ref{T_IntermittExcursionProc} below. The
latter shows that this type of asymptotic independence also holds in finite
measure situations as soon as the limit is not Gaussian.%

\vspace{0.2cm}%

In the situation of Proposition \ref{P_BasiXIntermittMap} b) set $\alpha:=1/p$
and let $(A_{n}^{(j)},B_{n}^{(j)})_{n\geq1}$ be the canonical normalizing
sequence for $[\alpha,c_{+}^{(j)},0]$, and let $S^{(j)}$ be the corresponding
$\alpha$-stable variable as in (\ref{Eq_TheFourierTransformOfG}), while
$B_{n}:=B_{n}^{(1)}$. Consider the processes $\mathsf{S}^{(j)[n]}%
=(\mathsf{S}_{t}^{(j)[n]})_{t\geq0}$ of excursions to $Z_{j}$ given by
\[
\mathsf{S}^{(j)[n]}:X\rightarrow\mathcal{D}([0,\infty),\mathbb{R})\text{,
\quad}\mathsf{S}_{t}^{(j)[n]}:=\frac{1}{B_{n}}\left(  \sum_{k=0}^{\left\lfloor
tn\right\rfloor -1}\varphi_{Y}^{(j)}\circ T_{Y}^{k}-\frac{\left\lfloor
tn\right\rfloor }{n}A_{n}^{(j)}\right)  \text{.}%
\]

\begin{theorem}
[\textbf{Asymptotically independent excursion processes}]%
\label{T_IntermittExcursionProc}Suppose that $T$ satisfies a)-d), and
(\ref{Eq_LocalBehAtFP}) with $p>2$. Then,
\[
(\mathsf{S}^{(0)[n]},\mathsf{S}^{(1)[n]})\overset{\nu}{\Longrightarrow
}(\mathsf{S}^{(0)},\mathsf{S}^{(1)})\text{ \quad in }(\mathcal{D}%
([0,\infty),\mathbb{R}^{2}),\mathcal{J}_{1})
\]
for any probability $\nu\ll\lambda$, where $\mathsf{S}^{(0)},\mathsf{S}^{(1)}$
are independent scalar $\alpha$-stable L\'{e}vy processes with $\alpha
:=1/p\in(0,2)$ and $\mathsf{S}_{1}^{(j)}\overset{d}{=}S^{(j)}$ characterized
by
\[
\mathbb{E}[e^{itS^{(j)}}]=e^{-c_{\alpha}c_{+}^{(j)}\left\vert t\right\vert
^{\alpha}(1-i\,\mathrm{sgn}(t)\,\omega(\alpha,t))}\text{, \quad}t\in
\mathbb{R}\text{.}%
\]

\end{theorem}

\begin{proof}
By the tail estimate (\ref{Eq_TailOfLocalReturnTimes}) in Proposition
\ref{P_BasiXIntermittMap}, the function $\varphi_{Y}^{(j)}\geq0$ is in the
domain of attraction of the $\alpha$-stable variable $S^{(j)}$, as
\[
\Pr[S^{(j)}>t]=(c_{+}^{(j)}+o(1))t^{-\alpha}\ell^{\ast}(t)\text{\quad
and}\quad\Pr[S^{(j)}<-t]=O(t^{-\alpha}\ell^{\ast}(t))\text{,}%
\]
with $\ell^{\ast}(t):=Ct^{1/p}A^{-1}(t)$, which corresponds to the case
$c_{-}=0$ in (\ref{Eq_DomainAttr}), meaning that $\underline{\beta}%
=c_{+}^{(j)}$ and $\beta=1$ in (\ref{Eq_TheFourierTransformOfG}).

On the other hand, $\{\varphi_{Y}^{(0)}>0\}\cap\{\varphi_{Y}^{(1)}%
>0\}=\varnothing$, and Theorem \ref{T_FSL_GMddim} applies.
\end{proof}%

\vspace{0.2cm}%

\begin{remark}
By routine arguments this theorem extends to more general Markovian interval
maps with finitely many neutral fixed (or periodic) points at which the map
satisfies conditions analogous to b) and c) above.
\end{remark}

\end{document}